\documentclass[a4paper,11pt,reqno]{amsart}

\usepackage[utf8]{inputenc}
\usepackage[T1]{fontenc}
\usepackage{lmodern}
\usepackage[english]{babel}
\usepackage{microtype}

\usepackage{amsmath,amssymb,amsfonts,amsthm,comment}
\usepackage{mathtools,accents}
\usepackage{mathrsfs}
\usepackage{aliascnt}
\usepackage{braket}
\usepackage{bm}
\usepackage{esint}
\usepackage{amsfonts}
\usepackage{csquotes}

\usepackage[a4paper,margin=3cm]{geometry}
\usepackage[citecolor=blue,colorlinks]{hyperref}
\usepackage[english,capitalize]{cleveref}

\usepackage{enumerate}
\usepackage{xcolor}
\usepackage{comment}

\usepackage{imakeidx} %Analytical index
\usepackage{xparse} %In order to use ExplSyntax (in older latex versions)
\usepackage{mathtools}
\usepackage{float}
\usepackage{comment}
\usepackage{nicefrac}
\usepackage{graphicx}
\usepackage{enumerate}

\usepackage[maxbibnames=99,backend=biber, sorting=nyt, doi=false, url=false]{biblatex}
\makeatletter
\g@addto@macro\@floatboxreset\centering
\makeatother

\makeatletter
\def\newaliasedtheorem#1[#2]#3{
  \newaliascnt{#1@alt}{#2}
  \newtheorem{#1}[#1@alt]{#3}
  \expandafter\newcommand\csname #1@altname\endcsname{#3}
}
\makeatother

\numberwithin{equation}{section}

\newtheoremstyle{slanted}{\topsep}{\topsep}{\slshape}{}{\bfseries}{.}{.5em}{}

\theoremstyle{plain}
\newtheorem{theorem}{Theorem}[section]
\newaliasedtheorem{proposition}[theorem]{Proposition}
\newaliasedtheorem{question}[theorem]{Question}
\newaliasedtheorem{lemma}[theorem]{Lemma}
\newaliasedtheorem{corollary}[theorem]{Corollary}
\newaliasedtheorem{counterexample}[theorem]{Counterexample}

\theoremstyle{definition}
\newaliasedtheorem{definition}[theorem]{Definition}
\newaliasedtheorem{example}[theorem]{Example}
\newaliasedtheorem{conjecture}[theorem]{Conjecture}

\theoremstyle{remark}
\newaliasedtheorem{remark}[theorem]{Remark}
\newcommand{\setN}{\mathbb{N}}
\newcommand{\N}{\mathbb{N}}

\newcommand{\setR}{\mathbb{R}}
\newcommand{\R}{\mathbb{R}}

\newcommand{\X}{{\rm X}}
\newcommand{\lip}{{\rm lip \,}}
\newcommand{\nchi}{{\raise.3ex\hbox{\(\chi\)}}}
\newcommand{\eps}{\varepsilon}

\let\phi\varphi

\newcommand{\abs}[1]{\left\lvert#1\right\rvert}

\DeclareMathOperator{\sn}{sn}

\newcommand{\di}{\mathop{}\!\mathrm{d}}

\newcommand{\Ch}{{\sf Ch}}
\newcommand{\st}{\ensuremath{\ :\ }} % Such that in formulas
\newcommand{\eqdef}{\ensuremath{\vcentcolon=}}

\DeclareMathOperator{\Lip}{LIP}
\newcommand{\haus}{\mathscr{H}}
\newcommand{\leb}{\mathscr{L}}

\newcommand{\dist}{\mathsf{d}}

\newcommand{\meas}{\mathfrak{m}}

\newcommand{\diam}{\mathrm{diam}}
\DeclareMathOperator{\CD}{CD}
\DeclareMathOperator{\RCD}{RCD}

\DeclareMathOperator{\Ric}{Ric}
\DeclareMathOperator{\Per}{Per}

\newfont{\tmpf}{cmsy10 scaled 2500}

\newcommand{\de}{\ensuremath{\,\mathrm d}} % Il de degli integrali (contiene lo spazio da mettere tra integranda e misura)

\keywords{Isoperimetric problem, Isoperimetric profile, Lower Ricci bounds, RCD space, Isoperimetric inequality.\\
The first author is partially supported by the European Research Council (ERC Starting Grant 713998 GeoMeG `\emph{Geometry of Metric Groups}').
The second author is supported by the Balzan project led by Luigi Ambrosio. The first and the third authors are partially supported by the INdAM - GNAMPA Project CUP\_E55F22000270001 ``Isoperimetric problems: variational and geometric aspects''.
The last author is supported by the European Research Council (ERC), under the European Union Horizon 2020 research and innovation programme, via the ERC Starting Grant  “CURVATURE”, grant agreement No. 802689. \\
Data sharing not applicable to this article as no datasets were generated or analysed during the current study.
}

\begin{document}

\title[Isoperimetry under lower Ricci bounds]{Asymptotic isoperimetry on non collapsed spaces with lower Ricci bounds}

\author[G. Antonelli]{Gioacchino Antonelli}\thanks{G. Antonelli: Scuola Normale Superiore, Piazza dei Cavalieri, 7, 56126 Pisa, Italy, ga2434@nyu.edu}

\author[E. Pasqualetto]{Enrico Pasqualetto}\thanks{E. Pasqualetto: Scuola Normale Superiore, Piazza dei Cavalieri, 7, 56126 Pisa, Italy, enrico.e.pasqualetto@jyu.fi}

\author[M. Pozzetta]{Marco Pozzetta}\thanks{M. Pozzetta: Dipartimento di Matematica e Applicazioni, Universit\`a di Napoli Federico II, Via Cintia, Monte S. Angelo, 80126 Napoli, Italy, marco.pozzetta@unina.it}

\author[D. Semola]{Daniele Semola}\thanks{D. Semola: Mathematical Institute, University of Oxford, Radcliffe Observatory, Andrew Wiles Building, Woodstock Rd, Oxford OX2 6GG, daniele.semola.math@gmail.com}

\maketitle

\begin{abstract}
  This paper studies sharp and rigid isoperimetric comparison theorems and asymptotic isoperimetric properties for small and large volumes on $N$-dimensional ${\rm RCD}(K,N)$ spaces $(X,\mathsf{d},\mathscr{H}^N)$. Moreover, we obtain almost regularity theorems formulated in terms of the isoperimetric profile and enhanced consequences at the level of several functional inequalities.\\  
  Most of our statements {seem to be} new even in the classical setting of smooth, non compact manifolds with lower Ricci curvature bounds. The synthetic theory plays a key role via compactness and stability arguments.
\end{abstract}

%\tableofcontents

\section{Introduction}

%\subsection*{Isoperimetry and lower Ricci curvature bounds}

The connection between lower Ricci curvature bounds on Riemannian manifolds and the isoperimetric problem is classical and in recent years it has attracted a lot of interest also in the realm of metric measure spaces with lower Ricci curvature bounds in synthetic sense. The goal of this paper is to explore a series of sharp isoperimetric comparison, stability and rigidity theorems in the setting of $N$-dimensional $\RCD(K,N)$ metric measure spaces $(X,\dist,\haus^N)$, for finite $N\in [1,\infty)$ and $K\in \setR$. 
Here $K\in \setR$ plays the role of (synthetic) lower bound on the Ricci curvature, $N\in [1,\infty)$ plays the role of (synthetic) upper bound on the dimension and $\haus^N$ indicates the $N$-dimensional Hausdorff measure.
\smallskip

Given an $\RCD(K,N)$ metric measure space $(X,\dist,\haus^N)$ we introduce the isoperimetric profile $I_X:[0,\haus^N(X))\to [0,\infty)$ by
\begin{equation}\label{eq:isoIntro}
    I_X(v):=\inf\left\{\Per(E)\, :\, E\subset X\, ,\, \, \haus^N(E)=v\right\}\, ,
\end{equation}
where we drop the subscript $X$ when there is no risk of confusion.
When $E\subset X$ attains the infimum in \eqref{eq:isoIntro} for $v=\haus^N(E)$, we call it an isoperimetric region.
Above, $\Per$ denotes the perimeter measure of a Borel set $E\subset X$ with finite perimeter. For the sake of this introduction we remark that the theory of sets of finite perimeter on $\RCD$ spaces is fully consistent with the Riemannian one.\\ 
In this setting we obtain:
\begin{itemize}
    \item sharp and rigid isoperimetric inequalities when $K=0$ and $(X,\dist,\haus^N)$ has Euclidean volume growth; 
    \item the precise asymptotic behaviour of the isoperimetric profile and of isoperimetric regions for small volumes in great generality, and for large volumes under the assumption that $K=0$ and $(X,\dist,\haus^N)$ has Euclidean volume growth;
    \item new global $\eps$-regularity theorems under lower Ricci curvature bounds formulated in terms of the isoperimetric profile;
    \item sharp and rigid comparison theorems for some spectral gap inequalities.
\end{itemize}

Many of the above results {seem to be} new even for smooth, non compact Riemannian manifolds with lower Ricci curvature bounds and for Alexandrov spaces with lower sectional curvature bounds. They answer to several open questions raised in \cite{Bayle03,NardulliOsorio20, BaloghKristaly}. The main technical contributions of the paper are a series of Geometric Measure Theory arguments in low regularity, leading to enhanced consequences about classical problems in Geometric Analysis under lower Ricci curvature bounds.

\subsection*{Main results}
The first main result of the paper is the characterization of rigidity for the sharp isoperimetric inequality on $\RCD(0,N)$ spaces endowed with the volume measure $\haus^N$ and with Euclidean volume growth.\\
We recall that the asymptotic volume ratio of an $\RCD(0,N)$ space $(X,\dist,\haus^N)$ is defined by
\begin{equation}
    \mathrm{AVR}(X,\dist,\haus^N):=\lim_{r\to \infty}\frac{\haus^N(B_r(p))}{\omega_Nr^N}\,,
\end{equation}
and it is independent of the base point $p\in X$. Here $\omega_N$ is the volume of the unit ball in $\mathbb R^N$. The space is said to have Euclidean volume growth if $\mathrm{AVR}(X,\dist,\haus^N)>0$.

\begin{theorem}\label{thm:IsoperimetricSharpRigidintro}
Let $(X,\dist,\haus^N)$ be an $\RCD(0,N)$ metric measure space with $\mathrm{AVR}(X,\dist,\haus^N)>0$ for some $N\in\setN$, $N\ge 2$, and let $E\subset X$ be a set of finite perimeter. Then 
\begin{equation}\label{eqn:SharpIsopintro} 
\Per(E)\ge N\omega_N^{\frac{1}{N}}\mathrm{AVR}(X,\dist,\haus^N)^{\frac{1}{N}}\left(\haus^N(E)\right)^{\frac{N-1}{N}}.
\end{equation}
Moreover equality holds for some $E$ with $\haus^N(E)\in(0,\infty)$ if and only if $X$ is isometric to a Euclidean metric measure cone (of dimension $N$) over an $\RCD(N-2,N-1)$ space and $E$ is isometric to a ball centered at one of the tips of $X$.
\end{theorem}

The novelty in \autoref{thm:IsoperimetricSharpRigidintro} is the complete characterization of the rigidity in \eqref{eqn:SharpIsopintro}. Moreover, even though the inequality \eqref{eqn:SharpIsopintro} was already known in the setting of \autoref{thm:IsoperimetricSharpRigidintro} {(see \cite{BaloghKristaly} after \cite{AgostinianiFogagnoloMazzieri,BrendleFigo,FogagnoloMazzieri})}, we will give a new proof tailored for the characterization of rigidity in this paper.
The rigidity statement of \autoref{thm:IsoperimetricSharpRigidintro} in the generality of $\RCD(0,N)$ spaces is fundamental in order to understand the asymptotic behaviour of the isoperimetric profile for small and large volumes even on a smooth non compact Riemannian manifold, see the results below. 
Furthermore, the characterization of the equality case in \autoref{thm:IsoperimetricSharpRigidintro} without further regularity assumptions besides the natural ones for the isoperimetric problem is new for Riemannian manifolds with dimension higher than $7$
%, see \autoref{thm:IsoperimetricSharpRigidintroSmooth} below, 
and completely new in the $\RCD$ setting.

The proof of \autoref{thm:IsoperimetricSharpRigidintro} combines the generalized existence of isoperimetric regions {obtained in \cite{AntonelliNardulliPozzetta} (see \autoref{thm:MassDecompositionINTRO})}, the sharp Laplacian comparison {for the distance function from isoperimetric sets obtained in \cite{AntonelliPasqualettoPozzettaSemolaFIRSThalf} (see \autoref{thm:Isoperimetriciintro})} and a limiting argument. {No assumption is made on the existence of isoperimetric regions for any volume. Nevertheless, the recently developed tools of Geometric Measure Theory on $\RCD$ spaces (see \cite{MoS21,AntonelliPasqualettoPozzettaSemolaFIRSThalf}) allow for a proof very much in the spirit of Gromov's original proof of the Lèvy-Gromov inequality \cite{Gromovmetric}.}
\medskip

In our companion paper \cite{AntonelliPasqualettoPozzettaSemolaFIRSThalf} we obtained sharp second order differential inequalities for the isoperimetric profile of $\RCD(K,N)$ spaces $(X,\dist,\haus^N)$ with volume of unit balls uniformly bounded from below.
Indeed, we proved that if $(X,\dist,\haus^N)$ is an $\RCD(K,N)$ space and there exists $v_0>0$ such that $\haus^N(B_1(x))\geq v_0$ for every $x\in X$, then
%\begin{enumerate}
 %   \item 
 the inequality
 \begin{equation}\label{eq:sharpdiffisop}
    -I''I\geq K+\frac{(I')^2}{N-1}
 \end{equation}
 holds in the viscosity sense on $(0,\haus^N(X))$, see \autoref{thm:BavardPansuIntro}. {The statement generalizes several results previously obtained for smooth Riemannian manifolds with Ricci curvature lower bounds, either in the compact case, or under uniform bounds on the geometry at infinity in the non compact case. See for instance \cite{BavardPansu86,MorganJohnson00,Bayle03,NiWangiso,MondinoNardulli16} and the introduction of \cite{AntonelliPasqualettoPozzettaSemolaFIRSThalf} for a more extended bibliography and a detailed comparison with the existing literature.}

{An elementary consequence of \eqref{eq:sharpdiffisop} that we explore in this paper is} that the {scale invariant isoperimetric profile}
\begin{equation}\label{monotone}
    v\mapsto \frac{I(v)}{v^{\frac{N-1}{N}}}
\end{equation}
is monotone decreasing on $\RCD(0,N)$ spaces $(X,\dist,\haus^N)$, see \autoref{cor:IsoperimetricProfileRCD0N}. 

The monotonicity formula for the {scale invariant isoperimetric profile} on spaces with non negative Ricci curvature has a series of consequences for other functional inequalities that can be obtained with the classical monotone rearrangement technique and are investigated in \autoref{subsec:functineq}. {Several other geometric and functional inequalities could be studied with analogous techniques.}
\medskip

The limit for large volumes in \eqref{monotone} is explicitly determined by the sharp isoperimetric inequality \autoref{thm:IsoperimetricSharpRigidintro} and Bishop--Gromov's inequality. Indeed, it is easily seen that the limit vanishes when the space has not Euclidean volume growth.

A combination of the ideas introduced in \cite{AntBruFogPoz} with the tools developed in {our previous paper \cite{AntonelliPasqualettoPozzettaSemolaFIRSThalf} and several compactness and stability arguments that fully exploit the $\RCD$ theory} leads to the asymptotic description of 
{ the isoperimetric behaviour for large volumes on a subclass of $\RCD(0,N)$ spaces with Euclidean volume growth that includes Alexandrov spaces with non negative sectional curvature and Euclidean volume growth.

Let $(X,\dist,\meas,x)$ be a pointed $\RCD(0,N)$ space with $\mathrm{AVR}(X,\dist,\meas)>0$. For every sequence $\{r_i\}_{i\in\mathbb N}$ with $r_i\to+\infty$ the sequence of pointed metric measure spaces $\{(X,r_i^{-1}\dist,r_i^{-N}\meas,x)\}_{i\in\mathbb N}$ is precompact in the pointed measured Gromov--Hausdorff (pmGH) topology, see \autoref{def:GHconvergence}. Every pmGH limit of such a sequence is a metric cone, by a slight modification of the proof of \cite[Proposition 2.8]{DePhilippisGigli18} (see also \cite[Theorem 7.6]{ChCo0}). Any such limit is called an {\em asymptotic cone of $X$}.

By \emph{Euclidean metric measure cone of dimension $N$} over a metric measure space, we mean a $(0,N)$-cone over that metric measure space, according to \cite[Definition 5.1]{Ketterer15}, where the reference measure is $\haus^N$. Given a Euclidean metric measure cone $(X,\dist,\haus^N)$ that is an $\RCD(0,N)$ space, we call a \emph{tip} any point $x\in X$ such that 
\[
\vartheta[X,\dist,\haus^N,x]:=\lim_{r\to 0}\frac{\haus^N(B_r(x))}{\omega_Nr^N}=\mathrm{AVR}(X,\dist,\haus^N)\, .
\]
\begin{theorem}\label{thm:AlexandrovRigidIntro}
Let $k\geq 0$, and let $(X,\dist,\haus^N)=(\mathbb R^k\times \tilde X,\dist_{\mathbb R^k}\times \dist_{\tilde X},\haus^N)$ be an $\RCD(0,N)$ metric measure space with $\mathrm{AVR}(X,\dist,\haus^N)>0$, and such that $\tilde X$ splits no lines. Assume that no asymptotic cone of $\tilde X$ splits a line. Then
\begin{itemize}
    \item there exists $V_0\geq 0$ such that for every $V\geq V_0$ there exists an isoperimetric region of volume $V$;
    \item  let $\{E_i\}_{i\geq 1}$ be a sequence of isoperimetric sets such that $\haus^N(E_i)=V_i$ with $V_i\to +\infty$, and fix $o \in X$. Then up to subsequence and translations along the Euclidean factor $\mathbb R^k$, the sequence $\overline{E}_i \subset (X, V_i^{-\frac1N}\dist,V_i^{-1}\haus^N,o)$ converges in the Hausdorff distance to a closed ball of volume $1$ centered at one tip of an asymptotic cone of $(X,\dist,\haus^N)$;
    \item if there exists $V>0$ such that 
\begin{equation}\label{eq:RigidProfileAlexandrov}
I(V)=N\omega_N^{\frac1N}({\rm AVR}(X,\dist,\haus^N))^{\frac1N}V^{\frac{N-1}{N}}\, ,
\end{equation}
then $N\ge2$, and $(X,\dist,\haus^N)$ is a Euclidean metric measure cone over a compact $\RCD(N-2,N-1)$ metric measure space.
\end{itemize}

\end{theorem}

We stress that \emph{every} Alexandrov space of dimension $N$ with nonnegative curvature and Euclidean volume growth falls in the hypotheses of \autoref{thm:AlexandrovRigidIntro}, since \cite[Theorem 4.6]{AntBruFogPoz} holds with the same proof in the setting of Alexandrov spaces.\\
We remark that
%the last two items in \autoref{thm:AlexandrovRigidIntro} are new even in the setting of smooth Riemannian manifolds with non negative sectional curvature. Moreover,  
an additional smoothness assumption on the cross sections of the asymptotic cone at infinity of an open manifold with non negative sectional curvature is not sufficient to guarantee smooth convergence to the blow-down, after scaling. Hence the results in \cite{ChodoshEichmairVolkmann17}, dealing with manifolds that are $C^{2,\alpha}$ asymptotically conical, do not apply in the generality considered above. %See \cite{ColdingNaber13} for related examples.}
}
%}
\medskip

The isoperimetric behaviour for large volumes on spaces with non negative Ricci curvature and Euclidean volume growth is determined by the large scale geometry of the space and of its pointed limits at infinity. Dual to this statement, we prove that the isoperimetric behaviour for small volumes of any $\RCD(K,N)$ space with volume of unit balls uniformly bounded from below is tightly linked to its structure at small scale and to that of its pointed limits at infinity.
\smallskip

{
Given any $\RCD(K,N)$ metric measure space $(X,\dist,\haus^N)$ such that $\haus^N(B_1(x))>v_0>0$ for every $x\in X$, we shall say that a pointed $\RCD(K,N)$ space $(Y,\dist_Y,\haus^N,y)$ is a pmGH limit at infinity of $(X,\dist,\haus^N)$ if there exists a sequence $x_i\in X$ with $\dist(x,x_i)\to\infty$ for some $x\in X$ such that $(X,\dist,\haus^N,x_i)$ converge to $(Y,\dist_Y,\haus^N,y)$ as $i\to\infty$ in the pmGH topology, see \autoref{def:GHconvergence}.

}

\begin{theorem}\label{cor:AsymptoticProfileZero}
Let $(X,\dist,\haus^N)$ be an $\RCD(K,N)$ space with isoperimetric profile function $I_X$. Let us assume $\inf_{x\in X}\haus^N(B_1(x))\geq v_0>0$. Then:
\begin{enumerate}
\item It holds
\begin{equation}\label{eqn:AsymptoticIsoperimetricProfileAtZero}
    \lim_{v\to 0}\frac{I_X(v)}{v^{\frac{N-1}{N}}} = N\left(\omega_N\vartheta_{\infty,\mathrm{min}}\right)^{\frac 1N}\, ,
\end{equation}
    where
\begin{equation}\label{eqn:EQUA}    
    \vartheta_{\infty,\mathrm{min}}=\lim_{r\to 0}\inf_{x\in X}\frac{\haus^N(B_r(x))}{v(N,K/(N-1),r)}\, >\, 0\, 
\end{equation}    
  is the minimum of all the possible densities at any point in $X$ or in any pmGH limit at infinity of $X$, and $v(N,K/(N-1),r)$ denotes the volume of the ball of radius $r$ in the simply connected model space with constant sectional curvature $K/(N-1)$ and dimension $N$;
   \item 
   Let $E_i \subset X_i$ be a sequence of sets with $\haus^N(E_i)=:V_i\to0$ and $\Per(E_i)=I_X(V_i)$, where $(X_i,\dist_i,\haus^N_{\dist_i})$ is either $X$ or a pmGH limit at infinity of $X$.\footnote{Such sets always exist by \autoref{lem:IsoperimetricAtFiniteOrInfinite}.}
    Let $(X_\infty,\dist_\infty,\haus^N)$ be a pmGH limit point of the sequence $\{(X_i,\dist_i,\haus^N_{\dist_i})\}_{i\in\mathbb N}$. Then, up to subsequences, the sets $E_i$ converge in the Hausdorff sense to a point where $\vartheta_{\infty,\mathrm{min}}$ is realized.
    
    \item Let $E_i$ be as in the previous item and let $q_i\in E_i$. Then, up to subsequences, the sets 
    \[
    \overline{E}_i\subset (X_i,V_i^{-1/N}\dist_i,V_i^{-1}\haus^N_{\dist_i},q_i)\, ,
    \]
    converge in the Hausdorff sense to a ball of volume 1 centered in one tip of a Euclidean metric measure cone with opening $\vartheta_{\infty,\mathrm{min}}$ over an $\RCD(N-2,N-1)$ space.
\end{enumerate}
\end{theorem}

The above {seems to be new also} for smooth non compact manifolds with Ricci curvature and volumes of unit balls uniformly bounded from below. {The proof combines the outcomes of \cite{AntonelliPasqualettoPozzettaSemolaFIRSThalf} with \autoref{thm:IsoperimetricSharpRigidintro} and a series of compactness and stability arguments. It is not difficult to build examples of smooth Riemannian manifolds with Ricci curvature uniformly bounded from below and volumes of unit balls uniformly bounded from below such that the value in \eqref{eqn:EQUA} is strictly less than $1$, see \autoref{rm:exmin1}.}
\smallskip

Starting from \autoref{cor:AsymptoticProfileZero} we prove a converse of the almost Euclidean isoperimetric inequality in \cite{CavallettiMondinoalmost}, to the effect that an almost Euclidean scale invariant isoperimetric { profile} for some volume $v\in (0,v_{K,N}]$ implies Reifenberg flatness of the space below a uniform radius $r(v,K,N)>0$, see \autoref{thm:isoperimetricepsregularity} and \autoref{rm:Reif} for the precise statements.\\
When the lower Ricci curvature bound is strengthened to a two sided Ricci curvature bound, the gap phenomenon of the classical \cite{Anderson90,ChCo1} reflects into a gap phenomenon for the almost Euclidean isoperimetric inequality. Namely, there exists $\eps(n)>0$ such that if $(M^n,g)$ is a smooth Riemannian manifold with two sided Ricci curvature bounds, then
\begin{equation}
       \lim_{v\to 0}\frac{I(v)}{v^{\frac{n-1}{n}}} \ge  n\omega_n^{\frac 1n}-\eps(n)\, ,
\end{equation}
if and only if 
\begin{equation}\label{eq:liintro}
     \lim_{v\to 0}\frac{I(v)}{v^{\frac{n-1}{n}}} =  n\omega_n^{\frac 1n}\, .
\end{equation}
Moreover, \eqref{eq:liintro} holds if and only if the harmonic radius is uniformly bounded away from zero on $(M^n,g)$, see \autoref{cor:RicciflatEinstein}.\\
{In the setting of noncollapsed $\RCD(K,N)$ spaces, these} results {connect} isoperimetry, volume, and regularity. Indeed, globally and under almost non negative Ricci curvature assumptions, almost regularity, almost Euclidean lower volume bounds and almost Euclidean isoperimetric { profile} at one scale/volume are all equivalent and propagate down to the bottom scale/volume. 
%To the best of our knowledge, the implication from an almost Euclidean isoperimetric inequality for a fixed volume on non compact spaces with non negative Ricci curvature to almost regularity below a definite scale is completely new.
\medskip

We believe that the techniques introduced here and in the companion paper \cite{AntonelliPasqualettoPozzettaSemolaFIRSThalf} will be useful to deal in the future with the general case of $\RCD(K,N)$ metric measure spaces $(X,\dist,\meas)$ and, in particular, with smooth weighted Riemannian manifolds verifying Curvature-Dimension bounds. 
%However, this more general case is left to future investigations, due to some additional difficulties on which we already commented at the end of the
We refer to the introduction of \cite{AntonelliPasqualettoPozzettaSemolaFIRSThalf} for a discussion of the main additional difficulties that are encountered when the assumption $\meas=\haus^N$ is dropped.

\subsection*{Comparison with the previous literature}

We conclude this introduction with a brief comparison between our results and the previous literature about the isoperimetric problem under lower curvature bounds, without the aim of being comprehensive.
\begin{itemize}
\item In \cite{BaloghKristaly} a sharp isoperimetric inequality is obtained for $\CD(0,N)$ metric measure spaces with Euclidean volume growth. In particular, this setting covers the case of $\RCD(0,N)$ spaces $(X,\dist,\haus^N)$ with Euclidean volume growth considered in \autoref{thm:IsoperimetricSharpRigidintro}. However, our strategy is different from \cite{BaloghKristaly} and also from the previous proofs in the Riemannian setting in \cite{AgostinianiFogagnoloMazzieri,FogagnoloMazzieri} (working up to dimension $7$) and \cite{BrendleFigo} (based on the Alexandrov--Bakelmann--Pucci method). For the generalization of the sharp isoperimetric inequality in the case of $\mathrm{MCP}(0,N)$ spaces with Euclidean volume growth, we point out the recent \cite{CavallettiManini}. 
       \item The setting of $\RCD(0,N)$ spaces $(X,\dist,\haus^N)$ recovers in particular many of the results for Euclidean convex cones treated in \cite{RitRosales04} and for cones with non negative Ricci curvature considered in \cite{MorganRitore02}.
    \item The results of the present paper recover, in a more general setting, many of the results proved in \cite{LeonardiRitore} for unbounded Euclidean convex bodies of uniform geometry, { and in \cite{MondinoNardulli16} for smooth Riemannian manifolds with uniform $C^0$ controls on the geometry at infinity.\\}
    The setting of $\RCD(K,N)$ spaces $(X,\dist,\haus^N)$ includes more in general convex subsets of smooth Riemannian manifolds with Ricci curvature bounded from below, regardless of any compactness assumption and regularity assumption on the boundary. Compact convex subsets of Riemannian manifolds with Ricci curvature lower bounds have been considered in \cite{BayleRosales}. 
\end{itemize}

\subsection*{Addendum}
This is the second of two companion papers, together with \cite{AntonelliPasqualettoPozzettaSemolaFIRSThalf}. The joint version of the two papers originally appeared on arXiv in \cite{ConcavitySharpAPPS}. In the first one, whose main results can be found in \autoref{subsec:FirstHalf}, we proved, in the setting of $N$-dimensional $\RCD(K,N)$ spaces with volumes of unit balls uniformly bounded from below, 
\begin{itemize}
     \item sharp second-order differential inequalities for the isoperimetric profile, corresponding to equalities on the model spaces with constant sectional curvature, see \autoref{thm:BavardPansuIntro};
    \item a sharp Laplacian comparison theorem for the distance function from isoperimetric boundaries, see \autoref{thm:Isoperimetriciintro}.
\end{itemize}
Several consequences of the above results play a fundamental role for the developments of the present paper.
\medskip

We mention that, a few months after the appearance of \cite{ConcavitySharpAPPS} on arXiv, Cavalletti and Manini generalized \autoref{thm:IsoperimetricSharpRigidintro}, see in particular \cite[Theorem 1.4 and Theorem 1.5] {CavallettiManiniRigid}, with a different technique. %Employing different techniques with respect to ours, 
Their main result is that a {\em bounded} set $E$ that saturates the sharp isoperimetric inequality \eqref{eqn:SharpIsopintro} in an essentially non branching $\CD(0,N)$ metric measure space $(X,\dist,\meas)$ with Euclidean volume growth must be a ball centered at some point $o$. Under the same assumptions, they prove that the space is a cone with respect to $o$ in a measure theoretic sense. If $(X,\dist,\meas)$ is an $\RCD(0,N)$ space, then the volume cone implies metric cone theorem (\cite{DePhilippisGigli16} after \cite{ChCo0}) shows that it is isomorphic to a metric measure cone over an $\RCD(N-2,N-1)$ space, therefore generalizing \autoref{thm:IsoperimetricSharpRigidintro} to the case of arbitrary reference measures.\\
The techniques in \cite{CavallettiManiniRigid} seem not sufficiently developed yet for the analysis of the isoperimetric profile pursued here and in our previous paper \cite{AntonelliPasqualettoPozzettaSemolaFIRSThalf}, though covering a more general setting.

\section{Preliminaries}

In this paper, by a \emph{metric measure space} (briefly, m.m.s.) we mean a triple \((X,\dist,\meas)\), where \((X,\dist)\) is a complete and separable
metric space, while \(\meas\geq 0\) is a boundedly-finite Borel measure on \(X\).
For any \(k\in[0,\infty)\), we denote by \(\haus^k\) the \emph{\(k\)-dimensional Hausdorff measure} on \((X,\dist)\).
We indicate with \(C(X)\) the space of all continuous functions \(f\colon X\to\R\) and \(C_b(X)\coloneqq\{f\in C(X)\,:\,f\text{ is bounded}\}\).
We denote by \({\rm LIP}(X)\subseteq C(X)\) the space of all Lipschitz functions, while \({\rm LIP}_{bs}(X)\) (resp.\ \({\rm LIP}_c(X)\)) stands for
the set of all those \(f\in{\rm LIP}(X)\) whose support \({\rm spt}f\) is bounded (resp.\ compact). More generally, we denote by
\({\rm LIP}_{loc}(X)\) the space of locally Lipschitz functions \(f\colon X\to\setR\). Given $f\in{\rm LIP}_{loc}(X)$,
\[
\lip f (x) \eqdef \limsup_{y\to x} \frac{|f(y)-f(x)|}{\dist(x,y)}
\]
is the \emph{slope} of $f$ at $x$, for any accumulation point $x\in X$, and $\lip f(x)\coloneqq 0$ if $x\in X$ is isolated.

We shall also work with the local versions of the above spaces: given \(\Omega\subseteq X\) open, we will consider the spaces
\({\rm LIP}_c(\Omega)\subseteq{\rm LIP}_{bs}(\Omega)\subseteq{\rm LIP}(\Omega)\subseteq{\rm LIP}_{loc}(\Omega)\). We underline
that by \({\rm LIP}_{bs}(\Omega)\) we mean the space of all \(f\in{\rm LIP}(\Omega)\) having bounded support \({\rm spt}f\subseteq\Omega\)
that verifies \(\dist({\rm spt}f,\partial\Omega)>0\).

\subsection{Convergence and stability results}

We introduce the pointed measured Gromov--Hausdorff convergence already in a proper realization even if this is not the general definition. Nevertheless, the simplified definition of Gromov--Hausdorff convergence via a realization is equivalent to the standard definition of pmGH convergence in our setting, because in the applications we will always deal with locally uniformly doubling measures, see \cite[Theorem 3.15 and Section 3.5]{GigliMondinoSavare15}. The following definition is taken from the introductory exposition of \cite{AmbrosioBrueSemola19}.

\begin{definition}[pGH and pmGH convergence]\label{def:GHconvergence}
A sequence $\{ (X_i, \dist_i, x_i) \}_{i\in \N}$ of pointed metric spaces is said to converge in the \emph{pointed Gromov--Hausdorff topology, in the $\mathrm{pGH}$ sense for short,} to a pointed metric space $ (Y, \dist_Y, y)$ if there exist a complete separable metric space $(Z, \dist_Z)$ and isometric embeddings
\[
\begin{split}
&\Psi_i:(X_i, \dist_i) \to (Z,\dist_Z), \qquad \forall\, i\in \N\, ,\\
&\Psi:(Y, \dist_Y) \to (Z,\dist_Z)\, ,
\end{split}
\]
such that for any $\eps,R>0$ there is $i_0(\varepsilon,R)\in\mathbb N$ such that
\[
\Psi_i(B_R^{X_i}(x_i)) \subset \left[ \Psi(B_R^Y(y))\right]_\eps,
\qquad
\Psi(B_R^{Y}(y)) \subset \left[ \Psi_i(B_R^{X_i}(x_i))\right]_\eps\, ,
\]
for any $i\ge i_0$, where $[A]_\eps\coloneqq \{ z\in Z \st \dist_Z(z,A)\leq \eps\}$ for any $A \subset Z$.

Let $\meas_i$ and $\mu$ be given in such a way $(X_i,\dist_i,\meas_i,x_i)$ and $(Y,\dist_Y,\mu,y)$ are m.m.s.\! If in addition to the previous requirements we also have $(\Psi_i)_\sharp\mathfrak{m}_i \rightharpoonup \Psi_\sharp \mu$ with respect to duality with continuous bounded functions on $Z$ with bounded support, then the convergence is said to hold in the \emph{pointed measured Gromov--Hausdorff topology, or in the $\mathrm{pmGH}$ sense for short}.
\end{definition}

We need to recall a generalized $L^1$-notion of convergence for sets defined on a sequence of metric measure spaces converging in the pmGH sense. Such a definition is given in \cite[Definition 3.1]{AmbrosioBrueSemola19}, and it is investigated in \cite{AmbrosioBrueSemola19} capitalizing on the results in \cite{AmbrosioHonda17}.

\begin{definition}[$L^1$-strong and $L^1_{\mathrm{loc}}$ convergence]\label{def:L1strong}
Let $\{ (X_i, \dist_i, \mathfrak{m}_i, x_i) \}_{i\in \N}$  be a sequence of pointed metric measure spaces converging in the pmGH sense to a pointed metric measure space $ (Y, \dist_Y, \mu, y)$ and let $(Z,\dist_Z)$ be a realization as in \autoref{def:GHconvergence}.

We say that a sequence of Borel sets $E_i\subset X_i$ such that $\mathfrak{m}_i(E_i) < +\infty$ for any $i \in \N$ converges \emph{in the $L^1$-strong sense} to a Borel set $F\subset Y$ with $\mu(F) < +\infty$ if $\mathfrak{m}_i(E_i) \to \mu(F)$ and $\chi_{E_i}\mathfrak{m}_i \rightharpoonup \chi_F\mu$ with respect to the duality with continuous bounded functions with bounded support on $Z$.

We say that a sequence of Borel sets $E_i\subset X_i$ converges \emph{in the $L^1_{\mathrm{loc}}$-sense} to a Borel set $F\subset Y$ if $E_i\cap B_R(x_i)$ converges to $F\cap B_R(y)$ in $L^1$-strong for every $R>0$.
\end{definition}

\begin{definition}[Hausdorff convergence]
Let $\{ (X_i, \dist_i, \mathfrak{m}_i, x_i) \}_{i\in \N}$  be a sequence of pointed metric measure spaces converging in the pmGH sense to a pointed metric measure space $ (Y, \dist_Y, \mu, y)$.

We say that a sequence of closed sets $E_i\subset X_i$ converges \emph{in Hausdorff distance} (or \emph{in Hausdorff sense}) to a closed set $F\subset Y$ if there holds convergence in Hausdorff distance in a realization $(Z,\dist_Z)$ of the pmGH convergence as in \autoref{def:GHconvergence}.
\end{definition}

{
In order to avoid confusion, we remark that the notions of convergence for sets that we recalled above do depend on the specific realization of the pmGH convergence of the ambient spaces into a metric space $(Z,\dist_Z)$. See for instance the discussion after Definition 3.23 in \cite{GigliMondinoSavare15}, where the notion of convergence of points is presented. However, this dependence does not affect any of the forthcoming arguments in the paper.}

\medskip

It is also possible to define notions of uniform convergence and $H^{1,2}$-strong and weak convergences for sequences of functions on a sequence of spaces $X_i$ converging in pointed measured Gromov--Hausdorff sense. We refer the reader to \cite{AmbrosioBrueSemola19, AmbrosioHonda17} for such definitions.

\subsection{\texorpdfstring{$\rm BV$}{BV} functions and sets of finite perimeter in metric measure spaces}
We begin with the definitions of \emph{function of bounded variation} and \emph{set of finite perimeter} in the setting of m.m.s.
\begin{definition}[$\rm BV$ functions and perimeter on m.m.s.]\label{def:BVperimetro}
Let $(X,\dist,\meas)$ be a metric measure space.  Given $f\in L^1_{loc}(X,\meas)$ we define
\[
|Df|(A) \eqdef \inf\left\{\liminf_i \int_A \lip f_i \de\meas \st \text{$f_i \in {\rm LIP}_{loc}(A),\,f_i \to f $ in $L^1_{loc}(A,\meas)$} \right\}\, ,
\]
for any open set $A\subseteq X$.
We declare that a function \(f\in L^1_{loc}(X,\meas)\) is of \emph{local bounded variation}, briefly \(f\in{\rm BV}_{loc}(X)\),
if \(|Df|(A)<+\infty\) for every \(A\subseteq X\) open bounded.
A function $f \in L^1(X,\meas)$ is said to belong to the space of \emph{bounded variation functions} ${\rm BV}(X)={\rm BV}(X,\dist,\meas)$ if $|Df|(X)<+\infty$. 

If $E\subseteq\X$ is a Borel set and $A\subseteq X$ is open, we  define the \emph{perimeter $\Per(E,A)$  of $E$ in $A$} by
\[
\Per(E,A) \eqdef \inf\left\{\liminf_i \int_A \lip u_i \de\meas \st \text{$u_i \in {\rm LIP}_{loc}(A),\,u_i \to \nchi_E $ in $L^1_{loc}(A,\meas)$} \right\}\, ,
\]
in other words \(\Per(E,A)\coloneqq|D\nchi_E|(A)\).
We say that $E$ has \emph{locally finite perimeter} if $\Per(E,A)<+\infty$ for every open bounded set $A$. We say that $E$ has \emph{finite perimeter} if $\Per(E,X)<+\infty$, and we denote $\Per(E)\eqdef \Per(E,X)$.
\end{definition}

Let us remark that when $f\in{\rm BV}_{loc}(X,\dist,\meas)$ or $E$ is a set with locally finite perimeter, the set functions $|Df|, \Per(E,\cdot)$ above are restrictions to open sets of Borel measures that we still denote by $|Df|, \Per(E,\cdot)$, see \cite{AmbrosioDiMarino14}, and \cite{Miranda03}.
\medskip

In the sequel, we shall frequently make use of the following \emph{coarea formula}, proved in \cite{Miranda03}.
\begin{theorem}[Coarea formula]\label{thm:coarea}
Let \((X,\dist,\meas)\) be a metric measure space. Let \(f\in L^1_{loc}(X)\) be given. Then for any open set
\(\Omega\subseteq X\) it holds that \(\R\ni t\mapsto \Per(\{f>t\},\Omega)\in[0,+\infty]\) is Borel measurable and satisfies
\[
|Df|(\Omega)=\int_\R \Per(\{f>t\},\Omega)\,\de t\, .
\]
In particular, if \(f\in{\rm BV}(X)\), then \(\{f>t\}\) has finite perimeter for a.e.\ \(t\in\setR\).
\end{theorem}
\begin{remark}[Semicontinuity of the total variation under $L^1_{loc}$-convergence]\label{rem:SemicontPerimeter}
Let $(X,\dist,\meas)$ be a metric measure space. We recall (cf.\ \cite[Proposition 3.6]{Miranda03}) that whenever $g_i,g\in L^1_{loc}(X,\meas)$
are such that $g_i\to g$ in $L^1_{loc}(X,\meas)$, for every open set $\Omega$ we have 
$$
|Dg|(\Omega)\leq \liminf_{i\to +\infty}|Dg_i|(\Omega)\, .
$$
\end{remark}

\subsection{Sobolev functions, Laplacians and vector fields in metric measure spaces}
The \emph{Cheeger energy} on a metric measure space \((X,\dist,\meas)\) is defined as the \(L^2\)-relaxation of the functional
\(f\mapsto\frac{1}{2}\int\lip^2 f_n\di\meas\) (see \cite{AmbrosioGigliSavare11} after \cite{Cheeger99}). Namely, for any function \(f\in L^2(X)\) we define
\[
\Ch(f)\coloneqq\inf\bigg\{\liminf_{n\to\infty}\int\lip^2 f_n\di\meas\;\bigg|\;(f_n)_{n\in\setN}\subseteq{\rm LIP}_{bs}(X),\,f_n\to f\text{ in }L^2(X)\bigg\}\, .
\]
The \emph{Sobolev space} \(H^{1,2}(X)\) is defined as the finiteness domain \(\{f\in L^2(X)\,:\,\Ch(f)<+\infty\}\) of the Cheeger energy.
The restriction of the Cheeger energy to the Sobolev space admits the integral representation \(\Ch(f)=\frac{1}{2}\int|\nabla f|^2\di\meas\),
for a uniquely determined function \(|\nabla f|\in L^2(X)\) that is called the \emph{minimal weak upper gradient} of \(f\in H^{1,2}(X)\).
The linear space \(H^{1,2}(X)\) is a Banach space if endowed with the Sobolev norm
\[
\|f\|_{H^{1,2}(X)}\coloneqq\sqrt{\|f\|_{L^2(X)}^2+2\Ch(f)}=\sqrt{\|f\|_{L^2(X)}^2+\||\nabla f|\|_{L^2(X)}^2},\quad\text{ for every }f\in H^{1,2}(X)\, .
\]
Following \cite{Gigli12}, when \(H^{1,2}(X)\) is a Hilbert space (or equivalently \(\Ch\) is a quadratic form) we say that the metric measure
space \((X,\dist,\meas)\) is \emph{infinitesimally Hilbertian}.
\medskip

For the rest of this paper, the infinitesimal Hilbertianity of \((X,\dist,\meas)\) will be our standing assumption.
The results of \cite{AmbrosioGigliSavare11_3} ensure that \({\rm LIP}_{bs}(X)\) is dense in \(H^{1,2}(X)\) with respect to the
norm topology. We define the bilinear mapping \(H^{1,2}(X)\times H^{1,2}(X)\ni(f,g)\mapsto\nabla f\cdot\nabla g\in L^1(X)\) as
\[
\nabla f\cdot\nabla g\coloneqq\frac{|\nabla(f+g)|^2-|\nabla f|^2-|\nabla g|^2}{2},\quad\text{ for every }f,g\in H^{1,2}(X)\, .
\]
Given \(\Omega\subseteq X\) open, we define the \emph{local Sobolev space with Dirichlet boundary conditions} \(H^{1,2}_0(\Omega)\) as the closure
of \({\rm LIP}_{bs}(\Omega)\) in \(H^{1,2}(X)\). Notice that \(H^{1,2}_0(X)=H^{1,2}(X)\). Moreover, we declare that a given function \(f\in L^2(\Omega)\)
belongs to the \emph{local Sobolev space} \(H^{1,2}(\Omega)\) provided \(\eta f\in H^{1,2}(X)\) holds for every \(\eta\in{\rm LIP}_{bs}(\Omega)\) and
\[
|\nabla f|\coloneqq{\rm ess\,sup}\big\{\chi_{\{\eta=1\}}|\nabla(\eta f)|\;\big|\;\eta\in{\rm LIP}_{bs}(\Omega)\big\}\in L^2(X)\, ,
\]
where by \({\rm ess\,sup}_{\lambda\in\Lambda}f_\lambda\) we mean the essential supremum of a set
\(\{f_\lambda\}_{\lambda\in\Lambda}\) of measurable functions.

\begin{definition}[Local Laplacian]
Let \((X,\dist,\meas)\) be an infinitesimally Hilbertian space and \(\Omega\subseteq X\) an open set. Then we say that a function
\(f\in H^{1,2}(\Omega)\) has \emph{local Laplacian} in \(\Omega\), \(f\in D(\Delta,\Omega)\) for short, provided there exists a (uniquely determined)
function \(\Delta f\in L^2(\Omega)\) such that
\[
\int_\Omega g\Delta f\di\meas=-\int_\Omega\nabla g\cdot\nabla f\di\meas,\quad\text{ for every }g\in H^{1,2}_0(\Omega)\, .
\]
For brevity, we write \(D(\Delta)\) instead of \(D(\Delta,X)\).
\end{definition}
More generally, we will work with functions having a measure-valued Laplacian in some open set:
\begin{definition}[Measure-valued Laplacian]
Let \((X,\dist,\meas)\) be an infinitesimally Hilbertian space and \(\Omega\subseteq X\) an open set. Then we say that a function
\(f\in H^{1,2}(\Omega)\) has \emph{measure-valued Laplacian} in \(\Omega\), \(f\in D(\mathbf\Delta,\Omega)\) for short, provided there
exists a (uniquely determined) locally finite measure \(\mathbf\Delta f\) on \(\Omega\) such that
\[
\int_\Omega g\mathbf\Delta f\coloneqq\int_\Omega g\di\mathbf\Delta f=-\int_\Omega\nabla g\cdot\nabla f\di\meas\, ,\quad\text{ for every }g\in{\rm LIP}_{bs}(\Omega)\, .
\]
For brevity, we write \(D(\mathbf\Delta)\) instead of \(D(\mathbf\Delta,X)\). Moreover, given functions \(f\in{\rm LIP}(\Omega)\cap H^{1,2}(\Omega)\)
and \(\eta\in C_b(\Omega)\), we say that \emph{\(\mathbf\Delta f\leq\eta\) in the distributional sense} provided \(f\in D(\mathbf\Delta,\Omega)\) and
\(\mathbf\Delta f\leq\eta\meas\).
\end{definition}
The above two notions of Laplacian are consistent, in the following sense: given any \(f\in H^{1,2}(\Omega)\), it holds that \(f\in D(\Delta,\Omega)\)
if and only if \(f\in D(\mathbf\Delta,\Omega)\), \(\mathbf\Delta f\ll\meas\) and \(\frac{\di\mathbf\Delta f}{\di\meas}\in L^2(\Omega)\). If this is
the case, then we also have that the \(\meas\)-a.e.\ equality \(\Delta f=\frac{\di\mathbf\Delta f}{\di\meas}\) holds.

\subsection{Geometric Analysis on \texorpdfstring{$\RCD$}{RCD} spaces}\label{sec:RCD}

The focus of this paper will be on $\RCD(K,N)$ metric measure spaces $(X,\dist,\meas)$, i.e. infinitesimally Hilbertian metric measure spaces with Ricci curvature bounded from below and dimension bounded from above, in synthetic sense.\\
The Riemannian Curvature Dimension condition $\RCD(K,\infty)$ was introduced in \cite{AmbrosioGigliSavare14} coupling the Curvature Dimension condition $\CD(K,\infty)$, previously proposed in \cite{Sturm1,Sturm2} and independently in \cite{LottVillani}, with the infinitesimally Hilbertian assumption.\\
The class $\RCD(K,N)$ was proposed in \cite{Gigli12}. The (a priori more general) $\RCD^*(K,N)$ condition was thoroughly analysed in \cite{ErbarKuwadaSturm15} and (subsequently and independently) in \cite{AmbrosioMondinoSavare15} (see also \cite{CavallettiMilmanCD} for the equivalence between $\RCD^*$ and $\RCD$ in the case of finite reference measure).
\medskip
 
Below we recall some of the main properties of $\RCD$ spaces that will be relevant for our purposes.
\medskip

The $\RCD(K,N)$ condition is compatible with the smooth notion. In particular, smooth $N$--dimensional Riemannian manifolds with Ricci curvature bounded from below by $K$ endowed with the canonical volume measure are $\RCD(K,N)$ spaces. Smooth Riemannian manifolds with smooth and convex boundary (i.e. non negative second fundamental form with respect to the interior unit normal) are also included in the theory.\\ 
Moreover, $N$-dimensional Alexandrov spaces with sectional curvature bounded from below by $K/(N-1)$ endowed with the $N$--dimensional Hausdorff measure are $\RCD(K,N)$ spaces.  
\medskip

	A fundamental property of $\RCD(K,N)$ spaces is the stability with respect to pmGH-convergence, meaning that a pmGH-limit of a sequence of (pointed) $\RCD(K_n,N_n)$ spaces for some $K_n\to K$ and $N_n\to N$ is an $\RCD(K,N)$ metric measure space.
\medskip

Let us define
\[
\sn_K(r) := \begin{cases}
(-K)^{-\frac12} \sinh((-K)^{\frac12} r) & K<0\, ,\\
r & K=0\, ,\\
K^{-\frac12} \sin(K^{\frac12} r) & K>0\, .
\end{cases}
\]
We denote by $v(N,K,r)$ the volume of the ball of radius $r$ in the (unique) simply connected Riemannian manifold of sectional curvature $K$ and dimension $N$, and by $s(N, K, r)$ the surface measure of the boundary of such a ball. In particular $s(N,K,r)=N\omega_N\mathrm{sn}_K^{N-1}(r)$ and $v(N,K,r)=\int_0^rN\omega_N\mathrm{sn}_K^{N-1}(r)\de r$, where $\omega_N$ is the Euclidean volume of the Euclidean unit ball in $\mathbb R^N$.
\medskip

For an arbitrary $\CD((N-1)K,N)$ space $(\X,\dist,\meas)$ the classical Bishop--Gromov volume comparison holds. More precisely, for any $x\in\X$, the function $\meas(B_r(x))/v(N,K,r)$ is nonincreasing in $r$ and the function $\Per(B_r(x))/s(N,K,r)$ is essentially nonincreasing in $r$, i.e., the inequality
\[
\Per(B_R(x))/s(N,K,R) \le \Per(B_r(x))/s(N,K,r)\, ,
\]
holds for almost every radii $R\ge r$, see \cite[Theorem 18.8, Equation (18.8), Proof of Theorem 30.11]{VillaniBook}. Moreover, it holds that 
\[
\Per(B_r(x))/s(N, K, r)\leq \meas(B_r(x))/v(N,K,r)\, ,
\]
for any $r>0$. The last inequality follows from the monotonicity of the volume and perimeter ratios together with the coarea formula on balls.

\medskip
	The Bishop--Gromov inequality implies that $\RCD(K,N)$ spaces are locally uniformly doubling. Then Gromov's precompactness theorem guarantees that any sequence of $\RCD(K,N)$ spaces $(X_n,\dist_n,\meas_n,x_n)$ such that $0<\inf_n\meas_n(B_1(x_n))<\sup_n\meas_n(B_1(x_n))<\infty$ is precompact with respect to the pointed measured Gromov-Hausdorff convergence.

\medskip

For most of the results of this paper we will consider \(\RCD(K,N)\) spaces of the form \((X,\dist,\haus^N)\), for some \(K\in\setR\)
and \(N\in\N\). Notice that we are requiring that the dimension of the Hausdorff measure coincides with the upper dimensional bound
in the \(\RCD\) condition. These spaces are typically called \emph{non collapsed \(\RCD\) spaces} (\({\rm ncRCD}(K,N)\) spaces for short) or $N$--dimensional $\RCD(K,N)$ spaces
(see \cite{Kitabeppu17,DePhilippisGigli18,MondinoKapovitch}).
\medskip

When $(X_n,\dist_n,\haus^N,x_n)$ are $\RCD(K,N)$ spaces and $\inf_n\haus^N(B_1(x_n))>0$, up to subsequences they converge to some $N$-dimensional $\RCD(K,N)$ space $(Y,\dist_Y,\haus^N,y)$, which amounts to saying that the $N$--dimensional Hausdorff measures converge to the $N$--dimensional Hausdorff measure of the limit when there is no collapse. This is the so-called volume convergence theorem, originally proved in \cite{Colding97,ChCo1} for smooth manifolds and their limits and extended in \cite{DePhilippisGigli18} to the present setting.
\medskip

If $(\X,\dist,\haus^N)$ is an $\RCD((N-1)K,N)$ space, $\haus^N$-almost every point has a unique measured Gromov--Hausdorff tangent isometric to $\mathbb R^N$ (\cite[Theorem 1.12]{DePhilippisGigli18} after \cite{MondinoNaber}). Therefore by volume convergence
 \begin{equation}\label{eqn:VolumeConv}
 \lim_{r\to 0}\frac{\haus^N(B_r(x))}{v(N,K,r)}=\lim_{r\to 0}\frac{\haus^N(B_r(x))}{\omega_Nr^N}=1\, , \qquad \text{for $\haus^N$-almost every $x$}\, ,
 \end{equation}
 where $\omega_N$ is the volume of the unit ball in $\mathbb R^N$. Moreover, since the density function $x\mapsto \lim_{r\to 0}\haus^N(B_r(x))/\omega_Nr^N$ is lower semicontinuous (\cite[Lemma 2.2]{DePhilippisGigli18}), it is bounded above by the constant $1$. Hence $\haus^N(B_r(x))\leq v(N,K,r)$ for every $r>0$ and for every $x\in X$.
\medskip

\medskip
Let us now recall some results about sets of locally finite perimeter in $\RCD$ spaces. Given a Borel set \(E\subseteq X\) in an \(\RCD(K,N)\) space \((X,\dist,\haus^N)\) and any \(t\in[0,1]\), we denote by \(E^{(t)}\) the set
of \emph{points of density \(t\)} of \(E\), namely
\[
E^{(t)}\coloneqq\bigg\{x\in X\;\bigg|\;\lim_{r\to 0}\frac{\haus^N(E\cap B_r(x))}{\haus^N(B_r(x))}=t\bigg\}\, .
\]
The \emph{essential boundary} of \(E\) is defined as \(\partial^e E\coloneqq X\setminus(E^{(0)}\cup E^{(1)})\). We have that \(E^{(t)}\)
and \(\partial^e E\) are Borel sets. Furthermore, the \emph{reduced boundary} \(\mathcal F E\subseteq\partial^e E\) of a given set \(E\subseteq X\)
of finite perimeter is defined as the set of those points of \(X\) where the unique tangent to \(E\) (up to isomorphism) is the
half-space \(\{x=(x_1,\ldots,x_N)\in\setR^N\,:\,x_N>0\}\) in \(\setR^N\); see \cite[Definition 4.1]{AmbrosioBrueSemola19} for the precise
notion of convergence we are using. We point out that in the classical Euclidean framework this notion of reduced boundary is not fully consistent
with the usual one, since it allows for non-uniqueness of the blow-ups (in the sense that one can obtain different half-spaces when rescaling
along different sequences of radii converging to \(0\)).\\ 
It was proved in \cite{BruePasqualettoSemola} after \cite{Ambrosio02,AmbrosioBrueSemola19} that the perimeter measure \(\Per(E,\cdot)\)
can be represented as
\begin{equation}\label{eq:RepresentationPerimeter}
\Per(E,\cdot)=\haus^{N-1}|_{\mathcal F E}.
\end{equation}

As it is evident from \eqref{eq:RepresentationPerimeter}, the notion of perimeter that we are using does not charge the boundary of the space under
consideration, if any. Indeed, by the very definition, reduced boundary points for $E$ are regular points of the ambient space $(X,\dist,\haus^N)$.
We refer to \cite{DePhilippisGigli18,MondinoKapovitch,BrueNaberSemola20} for the relevant background about boundaries of $\RCD(K,N)$ spaces $(X,\dist,\haus^N)$ and just point out here that the notion is fully consistent with the case of smooth Riemannian manifolds and with the theory of Alexandrov spaces with sectional curvature bounded from below.

Moreover, we recall that, according to \cite[Proposition 4.2]{BPSGaussGreen}, 
\[
\mathcal F E=E^{(1/2)}=\bigg\{x\in X\;\bigg|\;\lim_{r\to 0}\frac{\haus^N(E\cap B_r(x))}{\haus^N(B_r(x))}=\frac{1}{2}\bigg\}\, ,
\quad\text{ up to }\mathcal H^{N-1}\text{-null sets}\, .
\]
\smallskip

The primary focus of this note will be isoperimetric sets, that as in the classical Riemannian setting are much more regular than general sets of finite perimeter.

\begin{definition}\label{def:MinimizerCompactVariations}
Let $(X,\dist,\meas)$ be a metric measure space. We say that a subset $E\subset X$ is a \textit{volume constrained minimizer for compact variations in $X$} if whenever $F\subset X$ is such that $E\Delta F\subset K\Subset X$, and $\meas(K\cap E)=\meas(K\cap F)$, then $\Per(E)\leq \Per(F)$.

We say that a subset $E\subset X$, with $\meas(E)<\infty$, is an \textit{isoperimetric set} whenever for any $F\subset X$ with $\meas(F)=\meas(E)$ we have that $\Per(E)\leq \Per (F)$. 

Notice that an isoperimetric set in $X$ is a fortiori a volume constrained minimizer for compact variations in $X$.
\end{definition}

Let us recall a topological regularity result for volume constrained minimizers borrowed from \cite{AntonelliPasqualettoPozzetta21}.

\begin{theorem}[{\cite[Theorem 1.3 and Theorem 1.4]{AntonelliPasqualettoPozzetta21}}]\label{thm:RegularityIsoperimetricSets}
    Let $(X,\dist,\haus^N)$ be an $\RCD(K,N)$ space with $2\leq N<+\infty$ natural number, $K\in\mathbb R$. Let $E$ be a volume constrained minimizer for compact variations in $X$. Then $E^{(1)}$ is open, $\partial^eE=\partial E^{(1)}$, and $\partial E^{(1)}$ is locally uniformly $(N-1)$-Ahlfors regular in $X$. 
    
    Assume further there exists $v_0>0$ such that $\haus^N(B_1(x))\geq v_0$ for every $x\in X$, and that $E\subset X$ is an isoperimetric region. Then $E^{(1)}$ is in addition bounded, and $\partial E^{(1)}$ is $(N-1)$-Ahlfors regular in $X$. 
\end{theorem}

In the following, when $E$ is an isoperimetric region in a space $X$ as in \autoref{thm:RegularityIsoperimetricSets}, we will always assume that $E$ coincides with its open bounded representative given by $E^{(1)}$.

The proof of \autoref{thm:RegularityIsoperimetricSets} builds on the top of a deformation lemma for general sets of finite perimeter, see \cite[Theorem 1.1 and Theorem 2.35]{AntonelliPasqualettoPozzetta21}.

\subsection{Asymptotic mass decomposition}
The statements below are proved in \cite[Theorem 1.1 and Proposition 4.1]{AntonelliNardulliPozzetta} building on top of \cite{Nar14, AFP21, AntonelliPasqualettoPozzetta21}. They describe the precise behaviour of a minimizing (for the perimeter) sequence of sets with constant volume in the setting of $N$-dimensional non compact $\RCD(K,N)$ spaces with uniformly bounded volumes of unit balls. The following propositions are at the core of several proofs of this paper and its companion \cite{AntonelliPasqualettoPozzettaSemolaFIRSThalf}.

\begin{theorem}[Asymptotic mass decomposition]\label{thm:MassDecompositionINTRO}
Let $(X,\dist,\haus^N)$ be a non compact $\RCD(K,N)$ space. Assume there exists $v_0>0$ such that $\haus^N(B_1(x))\geq v_0$ for every $x\in X$. Let $V>0$. For every minimizing (for the perimeter) sequence $\Omega_i\subset X$ of volume $V$, with $\Omega_i$ bounded for any $i$, up to passing to a subsequence, there exist an increasing and bounded sequence $\{N_i\}_{i\in\mathbb N}\subseteq \mathbb N$, disjoint finite perimeter sets $\Omega_i^c, \Omega_{i,j}^d \subset \Omega_i$, and points $p_{i,j}$, with $1\leq j\leq N_i$ for any $i$, such that
\begin{itemize}
    \item $\lim_{i} \dist(p_{i,j},p_{i,\ell}) = \lim_{i} \dist(p_{i,j},o)=\infty$, for any $j\neq \ell\le\overline N$ and any $o\in X$, where $\overline N:=\lim_i N_i <\infty$;
    \item $\Omega_i^c$ converges to $\Omega\subset X$ in the sense of finite perimeter sets, and we have $\haus^N(\Omega_i^c)\to_i \haus^N(\Omega)$, and $ \Per( \Omega_i^c) \to_i \Per(\Omega)$. Moreover $\Omega$ is a bounded isoperimetric region for its own volume in $X$;
    \item for every $j\le\overline N$, $(X,\dist,\haus^N,p_{i,j})$ converges in the pmGH sense  to a pointed $\RCD(K,N)$ space $(X_j,\dist_j,\haus^N,p_j)$. Moreover there are isoperimetric regions $Z_j \subset X_j$ such that $\Omega^d_{i,j}\to_i Z_j$ in $L^1$-strong and $\Per(\Omega^d_{i,j}) \to_i \Per (Z_j)$;
    \item it holds that
    \begin{equation}\label{eq:UguaglianzeIntro}
    I_{(X,\dist,\haus^N)}(V) = \Per(\Omega) + \sum_{j=1}^{\overline{N}} \Per (Z_j),
    \qquad\qquad
    V=\haus^N(\Omega) +  \sum_{j=1}^{\overline{N}} \haus^N(Z_j).
    \end{equation}
\end{itemize}
\end{theorem}

\begin{proposition}\label{prop:ProfileDecomposition}
Let $(X,\dist,\haus^N)$ be a non compact $\RCD(K,N)$ space. Assume there exists $v_0>0$ such that $\haus^N(B_1(x))\geq v_0$ for every $x\in X$. Let $\{p_{i,j} \st i \in \mathbb N\}$ be a sequence of points on $X$, for $j=1,\ldots,\overline{N}$ where $\overline{N}\in\N \cup \{+\infty\}$. Suppose that each sequence $\{p_{i,j}\}_i$ is diverging along $X$ and that $(X,\dist,\haus^N,p_{i,j})$ converges in the pmGH sense  to a pointed $\RCD(K,N)$ space $(X_j,\dist_j,\haus^N,p_j)$. Defining
\begin{equation}\label{eqn:GeneralizedIsoperimetricProfile}
I_{X\sqcup_{j=1}^{\overline N}X_j}(v):=\inf\left\{\Per(E)+\sum_{j=1}^{\overline N}\Per (E_j):E\subseteq X,E_j\subseteq X_j, \haus^N(E)+\sum_{j=1}^{\overline N}\haus^N(E_j)=v\right\},
\end{equation}
 it holds $I_{X\sqcup_{j=1}^{\overline N} X_j}(v) = I_X(v)$ for any $v>0$.
\end{proposition}

\subsection{Preliminary results from \texorpdfstring{\cite{AntonelliPasqualettoPozzettaSemolaFIRSThalf}}{[16]}}\label{subsec:FirstHalf}

The remaining auxiliary results we shall need are taken from \cite{AntonelliPasqualettoPozzettaSemolaFIRSThalf} and are collected below. We first need to introduce the following notation:
\begin{equation}
    s_{k,\lambda}(r):=\cos_k(r)-\lambda\sin_k(r)\, ,
\end{equation}
where 
\begin{equation}
   \cos_k''+k\cos_k=0\, ,\quad \cos_k(0)=1\, ,\quad \cos_k'(0)=0\, , 
\end{equation}
and 
\begin{equation}
    \sin_k''+k\sin_k=0\, ,\quad\sin_k(0)=0\, ,\quad \sin_k'(0)=1\, .
\end{equation}
Notice that $s_{k,-d}$ is a solution of
\begin{equation}
    v''+kv=0\, ,\quad v(0)=1\, ,\quad v'(0)=d\, .
\end{equation}
Moreover, $s_{0,\lambda}(r)=1-\lambda r$.

Finally, for $N>1$, $H\in\setR$ and $K\in\setR$, we introduce the Jacobian function 
\begin{equation}\label{eq:Jacobian}
\setR\ni r\mapsto J_{H,K,N}(r):=\left(\cos_{\frac{K}{N-1}}(r)+\frac{H}{N-1} \sin_{\frac{K}{N-1}}(r)\right)^{N-1}_+=\left(s_{\frac{K}{N-1},-\frac{H}{N-1}}(r)\right)^{N-1}_+\, .
\end{equation}
Notice that, when $K=0$ the expression for the Jacobian function simplifies into
\begin{equation}
\setR\ni r\mapsto J_{H,N}(r):=\left(1+\frac{H}{N-1} r\right)^{N-1}_+\, .
\end{equation}
%We stress that the function $J_{H,K,N}$ is precisely the one involved in the one-dimensional comparison of $\CD(K,N)$ densities, see \cite[Corollary 4.3]{Ketterer20}.

\subsubsection{Mean curvature barriers for isoperimetric sets}
It is well-known that if an isoperimetric set $E$ in a Riemannian manifolds has smooth boundary $\partial E$, then $\partial E$ is a hypersurface with constant mean curvature. In \cite{AntonelliPasqualettoPozzettaSemolaFIRSThalf} (after \cite{MoS21}) we generalize the previous fact at the level of isoperimetric sets in $\RCD(K,N)$ spaces $(X,\dist,\haus^N)$ in the sense of the following result.

\begin{theorem}[{\cite{AntonelliPasqualettoPozzettaSemolaFIRSThalf}}]\label{thm:Isoperimetriciintro}
Let $(X,\dist,\haus^N)$ be an $\RCD(K,N)$ metric measure space for some $K\in\setR$ and $N\ge 2$ and let $E\subset X$ be an isoperimetric set.
Then, denoting by $f$ the signed distance function from $\overline{E}$, there exists $c\in\setR$ such that
\begin{equation}\label{eq:sharplaplacianintro}
\boldsymbol{\Delta} f\ge  -(N-1)\frac{s'_{\frac{K}{N-1},\frac{c}{N-1}}\circ \left(-f\right)}{s_{\frac{K}{N-1},\frac{c}{N-1}}\circ \left(-f\right)}      \quad\text{on $E$ and }\quad
\boldsymbol{\Delta} f\le (N-1)\frac{s'_{\frac{K}{N-1},-\frac{c}{N-1}}\circ f}{s_{\frac{K}{N-1},-\frac{c}{N-1}}\circ f}    \quad\text{on $X\setminus \overline{E}$}\, ,
\end{equation}
in the sense of distributions.
\end{theorem}

It is easily checked that the signed distance function from a smooth set $E$ in a Riemannian manifold satisfies \eqref{eq:sharplaplacianintro} if and only if $\partial E$ has constant mean curvature $c$. We remark that the inequalities \eqref{eq:sharplaplacianintro} are also sharp, as
equalities are attained by balls in the model spaces with constant sectional curvature. We stress that \autoref{thm:Isoperimetriciintro} is proved without any additional assumption on the regularity of the ambient space nor of the isoperimetric set, except the fundamental topological regularity from \autoref{thm:RegularityIsoperimetricSets}, which is also needed for \eqref{eq:sharplaplacianintro} to make sense.\\
In the realm of $\RCD(K,N)$ spaces, \autoref{thm:Isoperimetriciintro} just yields the existence of some constant as in the statement, hence we are in position to give the following definition.

% \begin{remark}
% By slightly modifying the proof of \autoref{thm:Isoperimetrici}, one can obtain analogous versions when $E$ is a volume constrained minimizer for compact variations in an open set $\Omega\subset X$. Since we do not need such level of generality, we omit to treat it.

% Notice, moreover, that in the proof of \autoref{thm:Isoperimetrici} we rely on a property that is slightly more general than the fact that $E$ is a volume constrained minimizer for compact variations in $X$. Indeed, we relied only on the assumption that for every $x\in X$ there exists a radius $r_x$ such that $E$ is volume constrained minimizer for compact variations in $\overline B_{r_x}(x)$.
% \end{remark}

\begin{definition}[Mean curvature barriers for isoperimetric sets]
Let $(X,\dist,\haus^N)$ be an $\RCD(K,N)$ space, let $E\subset X$ be an isoperimetric set. We call any constant $c$ such that \eqref{eq:sharplaplacianintro} holds as a \emph{mean curvature barrier} for $\partial E$. 
\end{definition}
For discussions concerning the uniqueness of $c$ as in the previous definition, and comparison with the Riemannian setting, we refer the reader to \cite{AntonelliPasqualettoPozzettaSemolaFIRSThalf}. { We mention that an analogous notion of mean curvature barrier has been recently studied in \cite{Ketterer21}, along with stability and splitting properties related to sets possessing such barriers.}

Exploiting the coarea formula, the Laplacian bounds in \eqref{eq:sharplaplacianintro} are sufficient to imply the following Heintze--Karcher type estimates on the perimeter and volume of equidistant sets from an isoperimetric region. { We mention that Heintze--Karcher type estimates were also obtained in \cite{KettererHeintzeKarcher}, where a different notion of mean curvature, based on the localization technique, is considered.}

\begin{proposition}[{\cite{AntonelliPasqualettoPozzettaSemolaFIRSThalf}}]\label{prop:variationofarea}
Let us consider an $\RCD(K,N)$ metric measure space $(X,\dist,\haus^N)$ for some $K\in\setR$ and $N\ge 2$.
Let $E\subset X$ be an isoperimetric set, and let $c\in\setR$ be given by \autoref{thm:Isoperimetriciintro}. Then for any $t\ge 0$ it holds
\begin{equation}\label{eq:extareabd}
\Per(\{x\in X\, :\, \dist(x,\overline{E})\le t\})\le J_{c,K,N}(t) \Per(E)\, ,
\end{equation}
and, for any $t\ge 0$,
\begin{equation}\label{eq:intareabound}
\Per(\{x\in X\, :\, \dist(x,X\setminus E)\le t\})\le J_{-c,K,N}(t)\Per(E)\,,
\end{equation}
where we recall that the Jacobian function has been introduced in \eqref{eq:Jacobian}.

In particular, for any $t\ge 0$ it holds
\begin{equation}\label{eq:extvolbound}
\haus^N(\{x\in X\setminus E\, :\, \dist(x,\overline{E})\le t\})\le \Per(E)\int _0^tJ_{c,K,N}(r)\di r\, ,
\end{equation}
and, for any $t\ge 0$,
\begin{equation}\label{eq:intvolbound}
\haus^N(\{x\in E\, :\, \dist(x,X\setminus E)\le t\})\le \Per(E)\int _0^tJ_{-c,K,N}(r)\di r\, .
\end{equation}
\end{proposition}

\subsubsection{Fine properties of the isoperimetric profile}

One of the main results of \cite{AntonelliPasqualettoPozzettaSemolaFIRSThalf} consists in the derivation of sharp second-order differential inequalities satisfied by the isoperimetric profile of an $\RCD$ space, without any assumption on the existence of isoperimetric sets. {Previous results about second order differential inequalities for the isoperimetric profile under lower Ricci curvature bounds in the smooth setting can be found in \cite{BavardPansu86,Bayle03,MorganJohnson00,NiWangiso,MondinoNardulli16}. We refer to the introduction of \cite{AntonelliPasqualettoPozzettaSemolaFIRSThalf} for a more detailed comparison with the previous literature. }\\
For the next result to make sense, we recall in the next remark the basic continuity and positivity property of the profile function.

\begin{remark}\label{rem:LocalHolderProfile}
Let $(X,\dist,\haus^N)$ be an $\RCD(K,N)$ space. Assume that there exists $v_0>0$ such that $\haus^N(B_1(x))\geq v_0$ for every $x\in X$. Let $I:(0,\haus^N(X))\to \mathbb R$ be the isoperimetric profile of $X$. Then $I(v)>0$ for every $v\in (0,\haus^N(X))$ and $I$ is continuous.\\
The previous claim is proved in \cite{AntonelliPasqualettoPozzettaSemolaFIRSThalf}, however it easily follows by \autoref{thm:MassDecompositionINTRO} and by adapting the proof of \cite[Theorem 2]{FloresNardulli20}.
\end{remark}

In the next result, we shall say that a function $I:(0,\haus^N(X))\to (0,\infty)$ satisfies a second-order differential inequality in the viscosity sense if whenever $\phi:(x_0-\eps,x_0+\eps)\cap (0,\haus^N(X))\to \setR$ is a $C^2$ function with $\phi\le I$ on $(x_0-\eps,x_0+\eps)\cap (0,\haus^N(X))$ and $\phi(x_0)=I(x_0)$, then the corresponding inequality holds at $x_0$ with $\phi$ in place of $I$.

\begin{theorem}[{\cite{AntonelliPasqualettoPozzettaSemolaFIRSThalf}}]\label{thm:BavardPansuIntro}
Let $(X,\dist,\haus^N)$ be an $\RCD(K,N)$ space. Assume that there exists $v_0>0$ such that $\haus^N(B_1(x))\geq v_0$ for every $x\in X$. 

Let $I:(0,\haus^N(X))\to (0,\infty)$ be the isoperimetric profile of $X$. Then:
\begin{enumerate}
    \item the inequality
 \[
    -I''I\geq K+\frac{(I')^2}{N-1}\,\quad\text{holds in the viscosity sense on $(0,\haus^N(X))$}\,,
 \]
    \item  if $\psi:=I^{\frac{N}{N-1}}$ then 
\begin{equation}\label{eqn:Bayle}
    -\psi''\geq \frac{KN}{N-1}\psi^{\frac{2-N}{N}}\,\quad\text{holds in the viscosity sense on $(0,\haus^N(X))$}\, .
\end{equation}
\end{enumerate}
\end{theorem}

The previous \autoref{thm:BavardPansuIntro} implies several fine properties on the isoperimetric profile, as stated below.

\begin{corollary}[{\cite{AntonelliPasqualettoPozzettaSemolaFIRSThalf}}]\label{cor:FinePropertiesProfile}
% Questo corollario accorpa \autoref{prop:SharpenedConcavityAndLimit} \autoref{cor:EstimateDerivativeProfileAndSubadditivity} e \autoref
Let $(X,\dist,\haus^N)$ be an $\RCD(K,N)$ space with $N\ge 2$. Assume that there exists $v_0>0$ such that $\haus^N(B_1(x))\geq v_0$ for every $x\in X$. Let $I:(0,\haus^N(X))\to (0,\infty)$ be the isoperimetric profile of $X$.

Then the following assertions hold true.
\begin{enumerate}
    \item There exists $C:=C(K,N,v_0)>0$ and $v_1:=v_1(K,N,v_0)>0$ such that the function $\eta(v):=I^{\frac{N}{N-1}}(v)-Cv^{\frac{2+N}{N}}$ is concave on the interval $[0,v_1]$.
    %Moreover, if $N=2$, we can choose $C=-K$ and the claim holds on $[0,\haus^N(X)]$, if $\haus^N(X)<\infty$, or on $[0,\haus^N(X))$, if $\haus^N(X)=\infty$.\\
    As a consequence the function
    \[
    [0,\haus^N(X))\ni v\mapsto \frac{I(v)}{v^{\frac{N-1}{N}}},
    \]
    has a finite strictly positive limit as $v\to 0$.
    
    \item There exists \(\tilde C=\tilde C(K,N,v_0)>0\) such that the function \(\tilde\eta(v)\coloneqq I(v)-\tilde C v^{\frac{1+N}{N}}\) is
    concave on \([0,v_1]\).
    
    \item There exists $\eps \eqdef \eps(K,N,v_0)>0$ such that $I$ is strictly subadditive on $(0,\eps)$. Moreover, if $K=0$, then one can take $\eps=\haus^N(X)$.
    
    \item For any $0<V_1<\haus^N(X)$ there exist $\mathscr{C},\mathscr{L}>0$ depending on $K,N,v_0,V_1$ such that
    \begin{equation}
    \begin{split}
        v\mapsto I^{\frac{N}{N-1}}(v) - \mathscr{C}v^{\frac{2+N}{N}} &\qquad\text{is concave on $[0,V_1]$}\, ,\\
      v\mapsto  I^{\frac{N}{N-1}}(v) &\qquad\text{is $\mathscr{L}$-Lipschitz on $[0,V_1]$}\, .
    \end{split}
    \end{equation}
\end{enumerate}
\end{corollary}

\begin{remark}
We anticipate that the limit $\lim_{v\to 0} I(v)/v^{\frac{N-1}{N}}$ will be explicitly computed in \autoref{cor:AsymptoticProfileZero}, yielding an answer to Question 4 in \cite{NardulliOsorio20}.
\end{remark}

\subsubsection{Regularity and stability of isoperimetric sets}

The fine properties in \autoref{cor:FinePropertiesProfile} imply useful regularity properties on isoperimetric sets, some of them stated in the next proposition. {We refer to \cite{MondinoNardulli16} for analogous results in the case of smooth Riemannian manifolds with nonnegative Ricci curvature and uniform controls on the geometry at infinity.}

\begin{proposition}[{\cite{AntonelliPasqualettoPozzettaSemolaFIRSThalf}}]\label{prop:NewRegularityIsoperimetricSets}
% Questa Prop accorpa \autoref{cor:GHRnRegioniConnesse} e \autoref{lem:BoundDiam}
Let $(X,\dist,\haus^N)$ be an $\RCD(K,N)$ space with $N\geq 2$ and $K\le 0$. Let us assume that $\inf_{x\in X}\haus^N(B_1(x))\geq  v_0>0$.

Then the following assertions hold true.
\begin{enumerate}
    \item Letting $\varepsilon>0$ be such that the isoperimetric profile $I$ is strictly subadditive on $(0,\varepsilon)$\footnote{Such an $\varepsilon>0$ exists thanks to \autoref{cor:FinePropertiesProfile}.}, if $E$ is an isoperimetric region in $X$ with $\haus^N(E)< \varepsilon$, then $E$ is connected. If in addition $\haus^N$ is finite, then $E$ is simple (i.e.\ $E$ and $X\setminus E$ are indecomposable) and $E^{(0)}$ is connected.
    
    \item There exist constants \(\bar v=\bar v(K,N,v_0)>0\) and $C=C(K,N,v_0)>0$ such that, if $E\subseteq X$ is an isoperimetric region, then
    \begin{equation}
    \diam E\leq C\haus^N(E)^{\frac{1}{N}}\quad\text{ whenever }\haus^N(E)\leq\bar v\, .
    \end{equation}
    Moreover, if \(K=0\) and $A\coloneqq{\rm AVR}(X,\dist,\haus^N)>0$, then it holds that
    \begin{equation}\label{eq:BoundDiamBig}
    \diam E\leq\tilde C\haus^N(E)^{\frac{1}{N}}\quad\text{ for every isoperimetric region }E\subseteq X\, ,
    \end{equation}
    for some constant \(\tilde C=\tilde C(K,N,A)>0\).
\end{enumerate}
\end{proposition}

Another useful result implied by \autoref{cor:FinePropertiesProfile} is a stability theorem for sequences of isoperimetric sets. We can prove that for sequences of isoperimetric sets $E_i$ converging in $L^1$ to a limit set $F$, the convergence can be improved to Hausdorff convergence of both the sets and their topological boundary. {Notice that the statement is a well known consequence of the classical regularity theory in the case when the ambient space is a fixed Riemannian manifold, or when a sequence of Riemannian metrics on a given manifold converges in a sufficiently strong norm (see for example the argument in \cite[Theorem 2.2]{MorganJohnson00} or \cite[Part III]{MaggiBook}). Here we deal with lower regularity assumptions for the ambient spaces and a weaker notion of convergence.}

Observe that no uniform hypotheses on the mean curvature barriers for the $E_i$'s are assumed, instead any sequence of mean curvature barriers for the $E_i$'s converges to a mean curvature barrier for $F$.

\begin{theorem}[{\cite{AntonelliPasqualettoPozzettaSemolaFIRSThalf}}]\label{thm:ConvergenceBarriersStability}
Let $(X_i,\dist_i,\haus^N,x_i)$ be a sequence of $\RCD(K,N)$ spaces converging to $(Y,\dist_Y,\haus^N,y)$ in pmGH sense, and let $(Z,\dist_Z)$ be a space realizing the convergence. Assume that $\haus^N(B_1(p))\ge v_0>0$ for any $p \in X_i$ and any $i$. Let $E_i \subset X_i$, $F\subset Y$.

If $E_i$ is isoperimetric, $E_i\subset B_R(x_i)$ for some $R>0$ for any $i$, $c_i$ is a mean curvature barrier for $E_i$ for any $i$, $E_i\to F $ in $L^1$-strong, and $0<\lim_i \haus^N(E_i)< \lim_i \haus^N(X_i)$, then
    \begin{equation}
    \begin{split}
        \text{$F$ is isoperimetric}\, ,& \\
         |c_i| \le L& \qquad \text{for any $i$ large enough}\, ,\\
        |D\chi_{E_i}|\to |D\chi_F|  & \qquad\text{in duality with $C_{\rm bs}(Z)$}\, ,\\
        \partial E_i \to \partial F\, , \,\, 
        \overline{E}_i \to \overline{F}  & \qquad\text{in Hausdorff distance in $Z$}\, ,
    \end{split}
    \end{equation}
where $L=L(K,N,v_0,\lim_i \haus^N(E_i), \lim_i \haus^N(X_i))>0$. In particular, the mean curvature barriers $c_i$ converge up to subsequence to a mean curvature barrier for $F$.
\end{theorem}

We conclude with a technical tool which combines \autoref{thm:MassDecompositionINTRO} with the previously obtained fine properties on the isoperimetric profile and on the topology of isoperimetric sets. {We refer again to \cite{MondinoNardulli16} for analogous results in the setting of smooth Riemannian manifolds with nonnegative Ricci curvature and uniform bounds on the geometry at infinity.}

\begin{lemma}[{\cite{AntonelliPasqualettoPozzettaSemolaFIRSThalf}}]\label{lem:IsoperimetricAtFiniteOrInfinite}
Let $(X,\dist,\haus^N)$ be an $\RCD(K,N)$ space with $N\geq 2$. Let us assume that $\inf_{x\in X}\haus^N(B_1(x))\geq  v_0>0$. Let $\varepsilon>0$ be such that the isoperimetric profile $I$ is strictly subadditive on $(0,\varepsilon)$\footnote{Such an $\varepsilon>0$ exists thanks to \autoref{cor:FinePropertiesProfile}.}.
\begin{enumerate}
    \item Let $\{\Omega_i\}_{i\in\mathbb N}$ be a minimizing (for the perimeter) sequence of bounded finite perimeter sets of volume $v<\varepsilon$ in $X$. Then, if one applies \autoref{thm:MassDecompositionINTRO}, either $\overline N=0$, or $\overline N=1$ and $\haus^N(\Omega)=0$.
    
    \item Let $X_1,\ldots, X_{\overline{N}}$ be pmGH limits of $X$ along sequences of points $\{p_{i,j}\}_{i \in \setN}$, for $j=1,\ldots,\overline{N} \in \setN \cup\{+\infty\}$. Let $\Omega=E \cup \bigcup_{j=1}^{\overline{N}} E_j$, with $E\subset X, E_j \subset X_j$ be a set achieving the infimum in \eqref{eqn:GeneralizedIsoperimetricProfile} for some $v<\varepsilon$. Then exactly one component among $E,E_1,\ldots,E_{\overline{N}}$ is nonempty.
\end{enumerate}
In particular, for any $v<\varepsilon$ there is an $\RCD(K,N)$ space $(Y,\dist,\haus^N)$ which is either $X$ or a pmGH limit of $X$ along a sequence $\{p_i\}_i \subset X$, and a set $E \subset Y$ such that $\haus^N(E)=v$ and $I_X(v)=\Per(E)$.
\end{lemma}

\section{Non negatively curved spaces}

%{\color{red}
%We assume the reader familiar with the notion of metric measure cones, addressing the reader to \cite{Ketterer15} for further details. In particular, by \emph{Euclidean metric measure cone of dimension $N$} over a metric measure space, we mean $(0,N)$-cone over that metric measure space, according to \cite[Definition 5.1]{Ketterer15}, where the reference measure is $\haus^N$. Given a Euclidean metric measure cone $(X,\dist,\haus^N)$ that is an $\RCD(0,N)$ space, we call a \emph{tip} any point $x\in X$ such that 
%\[
%\vartheta[X,\dist,\haus^N,x]:=\lim_{r\to 0}\frac{\haus^N(B_r(x))}{\omega_Nr^N}=\mathrm{AVR}(X,\dist,\haus^N)\, .
%\]
%Notice that, in such a case, by Bishop--Gromov monotonicity, it holds that $\haus^N(B_r(x))=\mathrm{AVR}(X,\dist,\haus^N)\omega_Nr^N$ for every $r>0$, and $X$ is a metric cone over the point $x\in X$.\footnote{MP: Questo pezzo già lo si scrive nella intro (?). In tal caso anche 'opening' si potrebbe mettere nello stesso pezzetto della intro.}
%}

We define the \emph{opening} of an $\RCD(0,N)$ Euclidean metric measure cone of dimension $N$ to be the value of the density at any tip, see the discussion before \autoref{thm:AlexandrovRigidIntro}. Hence, the opening of the cone is equal to its $\mathrm{AVR}$.

\subsection{Rigid bounds on the inscribed radius and Kasue-type rigidity}

We explicitly consider the cleaner expressions for the bounds in \eqref{eq:sharplaplacianintro} under the assumption that $K=0$. Observe that on an $\RCD(0,N)$ space, by the direct analog on $\RCD(0,N)$ spaces of \cite[Proposition 2.18]{AntBruFogPoz}, isoperimetric sets of positive measure may exist only if unit balls in $(X,\dist,\haus^N)$ have volume uniformly bounded from below, otherwise the isoperimetric profile of $X$ identically vanishes. 
\medskip

{We remark that any non compact $\RCD(0,N)$ metric measure space $(X,\dist,\meas)$ has infinite total mass. The statement for smooth Riemannian manifolds is due to Calabi and Yau independently. We refer to \cite{Huangvolgr} for a generalization to metric measure spaces verifying the $\mathrm{MCP}(0,N)$ condition, a setting more general than the one considered here.}

\begin{proposition}\label{prop:C>0}
Let $(X,\dist,\haus^N)$ be a non compact $\RCD(0,N)$ metric measure space for some $N\ge 2$ and assume that $\haus^N(B_1(x))\ge v_0>0$ for any $x \in X$. Let $E\subset X$ be an isoperimetric set. Then there exists $c\in [0,\infty)$ such that, denoting by $f$ the signed distance function from $\overline{E}$, it holds
\begin{equation}\label{eq:sharpLap0}
\Delta f\le \frac{c}{1+\frac{c}{N-1}f}\, ,\quad\text{on $X\setminus \overline{E}$}\, ,\quad \Delta f\ge  \frac{c}{1+\frac{c}{N-1}f}\, ,\quad\text{on $E$}\, .
\end{equation}
Furthermore, if ${\rm AVR}(X,\dist,\haus^N)>0$, then $c>0$. Moreover, if $c=0$, the following holds.
Let $(\tilde X,\tilde\dist)$ be the completion of $X\setminus \overline{E}$ endowed with the intrinsic distance induced by $\dist$, and let $\dist'$ be the intrinsic distance induced by $\dist$ on $\partial E$\footnote{In such a way, two different connected components of $(\tilde X,\tilde\dist)$ or $(\partial E,\dist')$ have distance equal to $+\infty$.}. Then every connected component of $(\partial E,\dist',\haus^{N-1})$ is an $\RCD(0,N-1)$-space, and $(\partial E\times[0,+\infty),\dist'\times\dist_{\mathrm{eu}},\haus^{N-1}\otimes\leb^1)$ is isomorphic to $(\tilde X,\tilde\dist,\haus^N)$ as metric measure spaces, where $\dist_{\mathrm{eu}}$ is the Euclidean distance.
\end{proposition}

\begin{proof}

Let us prove that $c\geq 0$. If $c<0$ the first bound in \eqref{eq:sharpLap0} degenerates on the set $\{x\in X\setminus \overline E:\dist(x,E)>(1-N)/c\}$, which is nonempty since $E$ is bounded and $X$ is non compact. This gives a contradiction and thus $c\geq 0$.
\medskip

Now let us assume that $c=0$. By \autoref{rem:LocalHolderProfile} and \autoref{thm:BavardPansuIntro} we know that the isoperimetric profile $I$ is strictly positive and concave {(notice that $\haus^N(X)=\infty$, as we remarked above)}. Hence $I'(v)\ge0$ for any volume $v>0$ such that $I$ is differentiable at $v$, and then $I$ is nondecreasing. Letting $E_t\coloneqq \{x\in X\, :\, \dist(x,\overline{E})<t\}$ for any $t>0$, we deduce that
\begin{equation}
\Per(E_t) \ge I(\haus^N(E_t)) \ge I(\haus^N(E)) = \Per(E)\, .
\end{equation}
On the other hand, since $c=0$, we have that $\Per(E_t)\le \Per(E)$ by \eqref{eq:extareabd}. Therefore
\begin{equation}
    \Per(E)=\Per(E_t) \qquad\forall t\ge0\, ,
\end{equation}
and, as $\Per(E_t) = I(\haus^N(E_t)) $, the set $E_t$ is an isoperimetric region for any $t>0$. Moreover $\haus^N(E_t\setminus E) = t\Per(E)$ by coarea. Hence the isoperimetric profile is constant on $[\haus^N(E),\infty)$.\\ 
Denote by $\dist_\Omega$ the distance from some set $\Omega$. By \eqref{eq:sharpLap0} we know that $\Delta \dist_E \le 0$ on $X\setminus \overline{E}$. On the other hand $\dist_E = t- \dist_{X\setminus E_t}$ on $E_t\setminus E$, then
\begin{equation}
    \Delta \dist_E = \Delta\left(- \dist_{X\setminus E_t}\right) \ge 0 
    \qquad
    \text{on $E_t\setminus \overline{E}$}\, ,
\end{equation}
where the inequality follows from \eqref{eq:sharpLap0} applied on the isoperimetric set $E_t$, since we already know that any mean curvature barrier for $E_t$ is non negative. Since $t>0$ is arbitrary, we proved that
\begin{equation}
    \Delta \dist_E = 0 \qquad 
    \text{on $X\setminus \overline{E}$}.
\end{equation}
The isomorphic splitting in the last part of the statement then follows analogously as in the recent \cite[Theorem 1.4]{Ketterer21}, following the lines of \cite{KettererKitabeppuLakzian21}, extending the classical Riemannian result \cite[Theorem C]{Kasue83}. See Corollary 3.18, the beginning of Section 5.2, Corollary 5.6, and Theorem 5.10 in \cite{KettererKitabeppuLakzian21}.
\medskip

Let us show that, if $(X,\dist,\haus^N)$ is an $\RCD(0,N)$ space and $\mathrm{AVR}(X,\dist,\haus^N)>0$, then $c>0$. Indeed, in this case $X$ is non compact and by the previous discussion we have $c\geq 0$. Let us suppose for the sake of contradiction that $c=0$. From \eqref{eq:extvolbound}, we have that, for every $t>0$,
\begin{equation}
\haus^N(\{x\in X\setminus E\, :\, \dist(x,\overline{E})\le t\})\le t\Per(E)\, .
\end{equation}
Let us now fix $p\in E$, and let $D:=\diam E<\infty$. The diameter is finite because of \autoref{thm:RegularityIsoperimetricSets}. For every $t>D$, we have that $B_{t}(p)\setminus B_D(p)\subseteq \{x\in X\setminus E\, :\, \dist(x,\overline{E})\le t\}$. Hence
\begin{align}
\haus^N(\{x\in X\setminus E\, :\, \dist(x,\overline{E})\le t\})\geq & \haus^N(B_t(p)\setminus B_D(p))\\
\geq &\omega_N\mathrm{AVR}(X,\dist,\haus^N) t^N-\haus^N(B_D(p))\,,
\end{align}
which is a contradiction with \eqref{eq:extvolbound} for $t$ large enough, since $N\geq 2$.
\end{proof}

\begin{remark}
We refer to the recent \cite[Theorem 1.1]{Ketterer21} for the local counterpart of the rigidity statement in \autoref{prop:C>0}.
\end{remark}

\autoref{prop:C>0} implies an inscribed radius bound in terms of a mean curvature barrier $c$ for an isoperimetric set. The bound is classical in the Riemannian setting, see \cite{Kasue83}, and it has been recently extended to essentially non branching $\CD(0,N)$ metric measure spaces in \cite[Theorem 1.1]{Burtscheretal20} (with a formulation slightly different from ours). Moreover, the equality case can also be characterized, see \cite[Theorem 1.4]{Burtscheretal20}, in the $\RCD$ case.\\ 
We report here the details of the proof in our setting for the reader's convenience. Let us also point out that the result admits natural generalizations to arbitrary lower Ricci curvature bounds.
\medskip

We first need a well-known topological result (cf.\ \cite{ChCo2, MondinoKapovitch}) which can now be more directly deduced employing the recent \cite{DengNonBranching}.

\begin{lemma}\label{lem:PalleConnesse}
Let $(X,\dist,\haus^N)$ be an $\RCD(K,N)$ metric measure space for some $N\ge 2$. Let $B_r(x)\subset X$ be a fixed ball and let $x_0 \in B_r(x) \setminus\{x\}$. Then for $\haus^N$-almost every $y \in B_r(x)$ there exists a geodesic from $x_0$ to $y$ contained in $B_r(x)\setminus\{x\}$. In particular $B_r(x)\setminus\{x\}$ is connected. Moreover, if $E\subset X$ is open and connected and $x \in E$, then $E \setminus\{x\}$ is connected.
\end{lemma}

\begin{proof}
Let $Y\subset B_r(x)$ be the set of points $y \in B_r(x)$ such that any geodesic from $x_0$ to $y$ passes through $x$. Since $(X,\dist)$ is a length space, it is immediately checked that $\dist(x_0,y)=\dist(x_0,x)+\dist(x,y)$ for any $y \in Y$. Suppose by contradiction that $\haus^N(Y)>0$, then
\[
0<\haus^N(Y) = \int_0^r \haus^{N-1}(Y \cap \partial^e B_t(x)) \de t.
\]
Since $N\ge2$, this implies that there are $y_1,y_2 \in Y$ such that $\dist(x,y_1)=\dist(x,y_2) \in(0,r)$ and $y_1\neq y_2$. If $\gamma_0$ is a geodesic from $x_0$ to $x$ and $\gamma_1$ (resp. $\gamma_2$) is a geodesic from $x$ to $y_1$ (resp. $y_2$), joining $\gamma_0$ with $\gamma_1$ (resp. $\gamma_2$) yields a geodesic $\sigma_1$ (resp. $\sigma_2$) from $x_0$ to $y_1$ (resp. $y_2$). Then the couple $\sigma_1,\sigma_2$ is branching according to \cite[Definition 2.23]{DengNonBranching}, and this contradicts \cite[Theorem 1.3]{DengNonBranching}.

Hence $\haus^N(Y)=0$, and since balls are path-connected, this implies that $B_r(x)\setminus\{x\}$ is path-connected as well. {Indeed, for any $z,w\in B_r(x)\setminus\{x\}$ we can find a point $y\in B_r(x)\setminus\{x\}$ (actually a set of full measure of such points) such that any minimizing geodesic from $z$ to $y$ and any minimizing geodesic from $w$ to $y$ has image contained in $B_r(x)\setminus\{x\}$. The concatenation of any two of these geodesics is a continuous path from $z$ to $w$ with image contained in $B_r(x)\setminus\{x\}$.}

Finally, letting $E$ open and connected with $x \in E$, there exists a ball $B_r(x) \subset E$. If by contradiction $E\setminus\{x\}= A \sqcup B$ for two disjoint nonempty open sets $A,B$, then $B_r(x)\setminus\{x\} = (A \cap B_r(x)) \sqcup ( B \cap B_r(x))$. Hence we have that, for example, $B \cap B_r(x)= \emptyset$ and $B_r(x)\setminus\{x\} \subset A$. Then $A \cup \{x\}$ is open and contained in $E$ and it is disjoint from $B$. Moreover $(A\cup\{x\}) \sqcup B = E$, implying that $E$ is not connected.
\end{proof}

\begin{corollary}[Inscribed radius bound]\label{cor:inscribed radius bound}
Let $(X,\dist,\haus^N)$ be an $\RCD(0,N)$ metric measure space for some $N\ge 2$. Let $E\subset X$ be a bounded open set such that the signed distance function from $E$ satisfies \eqref{eq:sharpLap0} for some $c>0$. Then
\begin{equation}
\sup_{x\in E}\dist(x,X\setminus E)\le \frac{N-1}{c}\, .
\end{equation}
Moreover, if $E$ is connected, then equality holds if and only if $E$ is isometric to a ball of radius $\frac{N-1}{c}$ centered at one tip of some Euclidean metric measure cone of dimension $N$.
\end{corollary}

\begin{proof}
The inequality immediately follows for otherwise the second inequality in \eqref{eq:sharpLap0} degenerates.
Let us now assume that $E$ is connected and that equality holds, i.e., there exists a ball of maximal radius $B_{\frac{N-1}{c}}(x_0)\subset E$. In particular $\partial B_{\frac{N-1}{c}}(x_0)\cap \partial E\neq \emptyset$. For simplicity, denote $R\coloneqq \frac{N-1}{c}$ in this proof. By \eqref{eq:sharpLap0} we know that
\begin{equation}
    \Delta \dist_{X\setminus E} \le \frac{-c}{1-\frac{c}{N-1}\dist_{X\setminus E}}
    \qquad
    \text{on $E$.}
\end{equation}
On the other hand, letting $\dist_{x_0}$ be the distance from $x_0$, the Laplacian comparison theorem (see \cite[Corollary 5.15]{Gigli12}) gives that
\begin{equation}
    \Delta \dist_{x_0} \le \frac{N-1}{\dist_{x_0}} \qquad
    \text{on $E\setminus \{x_0\}$,}
\end{equation}
in the sense of distributions. Letting $F\coloneqq \dist_{X\setminus E}+ \dist_{x_0}$ we find
\begin{equation}\label{eq:LaplFInRad}
    \begin{split}
        \Delta F & \le \frac{N-1}{\dist_{x_0}} - \frac{(N-1)c}{N-1-c\,\dist_{X\setminus E}}= \frac{(N-1)c}{\dist_{x_0}(N-1-c\,\dist_{X\setminus E})} \left( \frac{N-1}{c} - F \right),
    \end{split}
\end{equation}
on $E\setminus\{x_0\}$. We observe that $F\ge R$ on $E\setminus\{x_0\}$. Indeed $F\ge \dist_{x_0}\ge R$ on $E\setminus B_R(x_0)$, and $F= \dist_{X\setminus E}+ \dist_{x_0}\ge R$ on $B_R(x_0)$ as well. Hence $\Delta F \le 0$ on $E\setminus\{x_0\}$. On the other hand $F\equiv R$ along a geodesic $\gamma$ from $x_0$ to a point in the boundary of $E$. Since $E\setminus\{x_0\}$ is connected by \autoref{lem:PalleConnesse}, by the strong maximum principle this implies that $F\equiv R$ on $E\setminus\{x_0\}$. In particular $E\setminus\{x_0\}\subset B_R(x_0)$, for otherwise $F>R$ at some point in $(E\setminus\{x_0\})\setminus \overline{B}_R(x_0)$, and then $E= B_R(x_0)$.

From now on, let us denote $E^*\coloneqq E\setminus\{x_0\}$.
Since $\Delta F=0$ on $E^*$, then equality holds in \eqref{eq:LaplFInRad}, and then $\Delta \dist_{X\setminus E} = \frac{-c}{1-\frac{c}{N-1}\dist_{X\setminus E}}$ on $E^*$. Hence by \cite[Corollary 4.16]{CM20} we get that
\begin{equation}\label{eq:Eqhrigidity}
    (\log h_\alpha)' = \Delta (-\dist_{X\setminus E}) |_{X_\alpha}= \frac{c}{1-\frac{c}{N-1}\dist_{X\setminus E}}\bigg|_{X_\alpha},
\end{equation}
along the corresponding geodesic $X_\alpha$, where $\{h_\alpha,X_\alpha\}_{\alpha\in Q}$ are given by the disintegration of $\haus^N$ with respect to the signed distance function from $B_R(x_0)$. As $X_\alpha$ is defined from $x_0$ up to some point in $\partial B_R(x_0)$ at least, then $h_\alpha$ is defined on some interval $[-R,b(X_\alpha)]$, for $b(X_\alpha)>0$, and then \eqref{eq:Eqhrigidity} reads
\begin{equation}
    \frac{h_\alpha'(t)}{h_\alpha(t)} = \frac{c}{1+\frac{c}{N-1}t}
    \qquad
    \text{for a.e. $t \in (-R,0]$,}
\end{equation}
and then for any $t \in (-R,0]$, for $\mathfrak{q}$-a.e. $\alpha$.
Solving for $h_\alpha$ yields
\begin{equation}
    h_\alpha(t)=h_\alpha(0)\left(1+\frac{c}{N-1}t \right)^{N-1}\qquad
    \text{for $t \in [-R,0]$.}
\end{equation}
By disintegration, for any $r \in(0,R)$ we deduce that
\begin{equation}\label{eq:VolumeGrowthRigidityInRad}
    \begin{split}
        \haus^N(B_r(x_0) ) 
        &= \int_{Q} \int_0^r h_\alpha( t- R) \de t \de \mathfrak{q}(\alpha) = \left( \int_{Q} h_\alpha(0)\de\mathfrak{q}(\alpha)\right) \left( \frac{c}{N-1}\right)^{N-1}\frac1N r^N.
    \end{split}
\end{equation}
The conclusion follows by the volume-cone to metric-cone theorem in \cite[Theorem 1.1]{DePhilippisGigli16}.
\end{proof}

\begin{corollary}[Barrier bounds]\label{cor:MeanCurvatureIso}
Let $(X,\dist,\haus^N)$ be an $\RCD(0,N)$ space for $N\ge 2$  and assume that $\haus^N(B_1(x))\ge v_0>0$ for any $x \in X$. Let $E$ be an isoperimetric region of volume $\haus^N(E)\in (0,\infty)$. Let $c\in [0,+\infty)$ be a barrier given from \autoref{thm:Isoperimetriciintro}, cf. \autoref{prop:C>0}. Then
\begin{equation}\label{eq:ww}
c\leq \frac{N-1}{N}\frac{\Per(E)}{\haus^N(E)}\, .
\end{equation}
Moreover, if $c= \tfrac{N-1}{N}\tfrac{\Per(E)}{\haus^N(E)}$, then $E$ is isometric to a ball of radius $\frac{N-1}{c}$ centered at one tip of some Euclidean metric measure cone of dimension $N$.

If also ${\rm AVR}(X,\dist,\haus^N)>0$, then
\begin{equation}\label{eq:zz}
    c \ge (N-1) \left(\frac{N\omega_N{\rm AVR}(X,\dist,\haus^N)}{\Per(E)} \right)^{\frac{1}{N-1}}\, .
\end{equation}
\end{corollary}

\begin{proof}
By \autoref{cor:inscribed radius bound} we have the bound on the distance from the complement 
\begin{equation}\label{eq:IsopProof2}
\sup_{x\in E}\dist(x,X\setminus E)\le \frac{N-1}{c}\, ,
\end{equation}
we can apply \eqref{eq:intvolbound}, substituting $K=0$, to obtain
\begin{equation}
\haus^N(\{x\in E\, :\, \dist(x,X\setminus E)\le r\})\le \Per(E)\int _0^r\left(1-\frac{c}{N-1}s\right)^{N-1}\di s\, ,
\end{equation}
for any $0<r<\frac{N-1}{c}$.
In particular
\begin{equation}
\haus^N(E)\le \Per(E)\int_0^{\frac{N-1}{c}}\left(1-\frac{c}{N-1}s\right)^{N-1}\di s=\Per(E)\frac{(N-1)}{cN},
\end{equation}
and \eqref{eq:ww} follows.

Now assume that $c= \tfrac{N-1}{N}\tfrac{\Per(E)}{\haus^N(E)}$. Suppose by contradiction that $\sup_{x\in E}\dist(x,X\setminus E)< \tfrac{N-1}{c}$. Then by \eqref{eq:intareabound} and \autoref{cor:inscribed radius bound} we have that
\begin{equation}
\begin{split}
\haus^N(E) &= \haus^N(\{x\in E\, :\, \dist(x,X\setminus E)\le (N-1)/c\})\\
&= \int_0^{\frac{N-1}{c}} \Per(\{x\in E\, :\, \dist(x,X\setminus E)\le r\}) \de r
\\
&<\Per(E)\int_0^{\frac{N-1}{c}} \left( 1 - \frac{c}{N-1}r\right)^{N-1} \de r 
= \Per(E) \frac{N-1}{cN} = \haus^N(E),
\end{split}
\end{equation}
which is impossible. Therefore $\sup_{x\in E}\dist(x,X\setminus E)= \tfrac{N-1}{c}$. Since $E$ is connected by \autoref{prop:NewRegularityIsoperimetricSets}, the rigidity part in \autoref{cor:inscribed radius bound} implies the claim.

Assume now that ${\rm AVR}(X,\dist,\haus^N)>0$ and let $c\in(0,\infty)$ be a mean curvature barrier, see \autoref{thm:Isoperimetriciintro} and \autoref{prop:C>0}. As a consequence of \eqref{eq:extvolbound} we have the following
\begin{align}
\haus^N(\{x\in X\setminus E\, :\, \dist(x,E)\le r\})\le &\Per(E)\int _0^r\left(1+\frac{c}{N-1}s\right)^{N-1}\di s\\
=&\Per(E)\frac{N-1}{Nc}\left[\left(1+\frac{cr}{N-1}\right)^N-1\right]\, ,
\end{align}
for any $0<r<\infty$.
Then we can study the asymptotics of the right hand side above as $r\to \infty$ to obtain
\begin{equation}
\Per(E)\frac{N-1}{Nc}\left[\left(1+\frac{cr}{N-1}\right)^N-1\right]
= \Per(E)\frac{c^{N-1}}{N(N-1)^{N-1}}r^N \left(1+ O\left(r^{-1}\right) \right).
\end{equation}
Since $E$ is bounded by \autoref{thm:RegularityIsoperimetricSets}, the Euclidean volume growth condition implies that{, letting $x_0\in E$ be any point, we have}
{
\begin{equation}
\begin{split}
\haus^N(\{x\in X\setminus E\, :\, \dist(x,E)\le r\}) 
&\ge \haus^N(\{x \in X \st \dist(x,E)\le r\}) - \haus^N(E)\\& \ge \haus^N(B_r(x_0)) - \haus^N(E)
\\&\ge \mathrm{AVR}(X,\dist,\haus^N)\omega_Nr^N - \haus^N(E)\\
&= \left(1 +O\left(r^{-N}\right) \right) \mathrm{AVR}(X,\dist,\haus^N)\omega_Nr^N 
\end{split}
\end{equation}
as $r\to+\infty$,} and then \eqref{eq:zz} follows.
\end{proof}

%{\color{red}
%\begin{remark}
%An argument similar to the one that we use to derive \eqref{eq:zz} has been recently exploited by Wang \cite{Wang21} to give an alternative proof of the sharp Willmore inequality on smooth bounded open sets in non negatively Ricci curved Riemannian manifolds with Euclidean volume growth obtained in \cite{AgostinianiFogagnoloMazzieri}.
%\end{remark}
%}

\subsection{Sharp and rigid isoperimetric inequalities}

In the following we give a new proof, tailored for $\RCD(0,N)$ spaces with reference measure $\haus^N$, of the sharp isoperimetric inequality under the Euclidean volume growth assumption.\\ 
The approach presented is suited for dealing with the rigidity case, thus extending the rigidity result for the sharp isoperimetric inequality treated in \cite{BrendleFigo, BaloghKristaly, AgostinianiFogagnoloMazzieri, FogagnoloMazzieri}. Notice also that the rigidity results in \cite{BrendleFigo, BaloghKristaly} need an a-priori hypothesis on the regularity of the boundary of the set $E$. Thus our approach to rigidity not only deals with the larger setting of $\RCD(0,N)$ spaces with reference measure $\haus^N$, but also recovers a slightly empowered version of the rigidity results in the smooth setting, without assuming any  regularity of the boundary. Compare with the discussion in \cite[Section 5.2]{BaloghKristaly}.

\begin{lemma}\label{lem:RigidityIsop}
Let $N\ge 2$. Let $(X,\dist,\haus^N)$ be an $\RCD(0,N)$ metric measure space with $\mathrm{AVR}(X,\dist,\haus^N)>0$, and let $E\subset X$ be a set of finite perimeter. If $E$ is an isoperimetric region, then 
\begin{equation}
\Per(E)\ge N\omega_N^{\frac{1}{N}}\left(\mathrm{AVR}(X,\dist,\haus^N)\right)^{\frac{1}{N}}\left(\haus^N(E)\right)^{\frac{N-1}{N}}\, .
\end{equation}
Moreover equality holds for some $E$ with $\haus^N(E)\in(0,\infty)$ if and only if $X$ is isometric to a Euclidean metric measure cone (of dimension $N$) over an $\RCD(N-2,N-1)$ space, and $E$ is isometric to a ball centered at one of the tips of $X$.
\end{lemma}

\begin{proof}
By \autoref{cor:MeanCurvatureIso} we get the lower bound
\begin{equation}\label{eq:IsopProof1}
\Per(E)\ge \max\left\{\frac{N\haus^N(E)c}{(N-1)}\, , \mathrm{AVR}(X,\dist,\haus^N)\omega_N\cdot \frac{N(N-1)^{N-1}}{c^{N-1}}\right\}.
\end{equation}
Hence
\begin{align}\label{eq:IsopProof}
\Per(E)\ge& \left(\frac{N\haus^N(E)c}{(N-1)}\right)^{\frac{N-1}{N}}\left( \mathrm{AVR}(X,\dist,\haus^N)\omega_N\cdot \frac{N(N-1)^{N-1}}{c^{N-1}}\right)^{\frac{1}{N}}\\
=& N\omega_N^{\frac{1}{N}}\left(\mathrm{AVR}(X,\dist,\haus^N)\right)^{\frac{1}{N}}\left(\haus^N(E)\right)^{\frac{N-1}{N}},
\end{align}
which is the sharp isoperimetric inequality.

\medskip

Now let us assume that $\haus^N(E)\in(0,\infty)$ and equality holds in \eqref{eq:IsopProof}. This implies that the two competitors in the right hand side of \eqref{eq:IsopProof1} are equal, that is
\begin{equation}
    \haus^N(E)= \omega_N \mathrm{AVR}(X,\dist,\haus^N)\left(\frac{N-1}{c} \right)^N.
\end{equation}
By \autoref{cor:inscribed radius bound} we know that $E$ contains a ball $B$ of radius $\tfrac{N-1}{c}$. By Bishop--Gromov monotonicity, the measure of $B$ satisfies $\haus^N(B)\ge \omega_N\mathrm{AVR}(X,\dist,\haus^N)\left(\tfrac{N-1}{c} \right)^N=\haus^N(E)$. As $B\subset E$, we conclude that $E=B$ is a metric ball in $X$. Let us write $E=B_{\frac{N-1}{c}}(x)$ for some $x$. By Bishop--Gromov monotonicity we obtain that
\[
\haus^N(B_R(x)) = \left(\frac{R}{\frac{N-1}{c}} \right)^N \haus^N\left(B_{\frac{N-1}{c}}(x)\right),
\]
for any $R\ge \tfrac{N-1}{c}$. As $\RCD(0,N)$ spaces are in particular $\RCD^*(0,N)$, we are in position to apply the rigidity result in \cite[Theorem 1.1]{DePhilippisGigli16}. Since $X$ is endowed with the Hausdorff measure $\haus^N$, we conclude that item (3) in \cite[Theorem 1.1]{DePhilippisGigli16} holds. Since $R\ge \tfrac{N-1}{c}$ is arbitrary, we get that $N\ge2$ and $X$ is isometric to a metric measure cone on a bounded $\RCD^*(N-2,N-1)$ space with finite measure, which is, in particular, also an $\RCD(N-2,N-1)$ space by \cite{CavallettiMilmanCD}.
The same rigidity result yields that $E$ is the ball centered at one of the tips of $X$.
\end{proof}

Though giving the rigidity of the sharp isoperimetric inequality in the setting of \autoref{lem:RigidityIsop}, the last argument seems to give an alternative proof of the sharp isoperimetric inequality only if we know a priori that isoperimetric regions exist for every volume on $X$. Nevertheless combining
\autoref{thm:MassDecompositionINTRO} and \autoref{lem:RigidityIsop} we can give an alternative proof of the sharp isoperimetric inequality in the setting of $\RCD(0,N)$ spaces $(X,\dist,\haus^N)$ with $\mathrm{AVR}(X,\dist,\haus^N)>0$, together with a characterization of the equality case.
In particular, we obtain an alternative proof of the sharp isoperimetric inequality in the setting of Riemannian manfiolds with non negative Ricci curvature and Euclidean volume growth.

\begin{proof}[Proof of \autoref{thm:IsoperimetricSharpRigidintro}]
Let $V:=\haus^N(E)$, and take $\Omega_i$ a minimizing sequence of bounded sets of volume $V$. In the setting of \autoref{thm:MassDecompositionINTRO}, we have that for every $1\leq j\leq \overline N$ the inequality $\mathrm{AVR}(X_j,\dist_j,\haus^N)\geq \mathrm{AVR}(X,\dist,\haus^N)$ holds as a consequence of the volume convergence from \cite{DePhilippisGigli18} and the monotonicity of Bishop--Gromov ratios. Hence, by using the latter inequality, together with \eqref{eq:UguaglianzeIntro}, the fact that $\Omega,Z_j$ are isoperimetric, and \autoref{lem:RigidityIsop}, we have
\begin{equation}\label{eq:EstimateProofRigid}
    \begin{split}
        \Per(E)&\geq I(V) = \Per(\Omega) + \sum_{j=1}^{\overline{N}} \Per (Z_j) \\ &\ge N\omega_N^{\frac{1}{N}}\bigg(\left(\mathrm{AVR}(X,\dist,\haus^N)\right)^{\frac{1}{N}}\left(\haus^N(\Omega)\right)^{\frac{N-1}{N}}+\\
        &\qquad+\sum_{j=1}^{\overline N}\left(\mathrm{AVR}(X_j,\dist_j,\haus^N)\right)^{\frac{1}{N}}\left(\haus^N(Z_j)\right)^{\frac{N-1}{N}}\bigg) \\
        &\geq N\omega_N^{\frac 1N}\mathrm{AVR}(X,\dist,\haus^N)^{\frac 1N}\left(\left(\haus^N(\Omega)\right)^{\frac{N-1}{N}}+\sum_{j=1}^{\overline N}\left(\haus^N(Z_j)\right)^{\frac{N-1}{N}}\right) \\
        &\geq N\omega_N^{\frac 1N}\mathrm{AVR}(X,\dist,\haus^N)^{\frac 1N}\left(\haus^N(\Omega)+\sum_{j=1}^{\overline N}\haus^N(Z_j)\right)^{\frac{N-1}{N}} \\
        &=N\omega_N^{\frac 1N}\mathrm{AVR}(X,\dist,\haus^N)^{\frac 1N}\left(\haus^N(E)\right)^{\frac{N-1}{N}}.
    \end{split}
\end{equation}
The rigidity part of the statement follows from \autoref{lem:RigidityIsop}.
\end{proof}

The rigidity part of \autoref{thm:IsoperimetricSharpRigidintro} allows to characterize isoperimetric sets in $\RCD(0,N)$ cones.

\begin{corollary}\label{cor:IsoperimetricOnCones}
Let $(X,\dist,\haus^N)$ be a Euclidean metric measure cone (of dimension $N$) over an $\RCD(N-2,N-1)$ space, for some $N\geq 2$. Let $\vartheta$ be the opening of the cone, i.e., the density at any tip. Then 
\begin{equation}\label{eqn:IsopOfaCone}
I_X(v)=N(\omega_N\vartheta)^{1/N}v^{\frac{N-1}{N}}, \qquad \text{for all $v>0$}\, .
\end{equation}
Moreover, all the isoperimetric regions in $X$ are balls centered at one of the tips.
\end{corollary}

% \begin{proof}
% It is readily seen that $\mathrm{AVR}(X,\dist,\haus^N)=\vartheta$. Hence $I_X(v)\geq N(\omega_N\vartheta)^{1/N}v^{\frac{N-1}{N}}$ for all $v>0$ by the sharp isoperimetric inequality \eqref{eqn:SharpIsop}. By a simple computation, any ball centered at one tip of the cone saturates the inequality in \eqref{eqn:SharpIsop}. Hence we get \eqref{eqn:IsopOfaCone} and balls centered at any tip of the cone are isoperimetric regions. The fact that they are the unique isoperimetric regions readily follows from the rigidity part of \autoref{thm:IsoperimetricSharpRigid}.
% \end{proof}

\subsection{Isoperimetric monotonicity on spaces with non negative Ricci curvature}

In this section we derive some further consequences of the sharp concavity properties of the isoperimetric profile for $\RCD(0,N)$ spaces $(X,\dist,\haus^N)$.
\medskip

Recall that in \cite{Huiskenvideo} a notion of isoperimetric cone angle for an $N$-dimensional Riemannian manifold $(M,g)$ with non negative Ricci curvature was proposed as 
\begin{equation*}
c_{\mathrm{iso}}(M,g):=\inf\left\{\frac{\left(\haus^{N-1}(\partial \Omega)\right)^{\frac{N}{N-1}}}{\left(N\omega_{N}^{\frac{1}{N}}\right)^{\frac{N}{N-1}}\haus^N(\Omega)}\, :\, \emptyset \neq \Omega\subset M\, , \quad\text{$\Omega$ open with smooth boundary}\right\}\, .
\end{equation*}
With this definition, the sharp isoperimetric inequality for smooth Riemannian manifolds with non negative Ricci curvature (and Euclidean volume growth) \cite{BrendleFigo,AgostinianiFogagnoloMazzieri,FogagnoloMazzieri} could be restated as
\begin{equation}
   c_{\mathrm{iso}}(M,g)\ge \left(\mathrm{AVR}(M,g)\right)^{\frac{1}{N-1}}\, ,
\end{equation}
{ where $\mathrm{AVR}(M,g)\eqdef {\rm AVR}(M,\dist,\haus^N)$ for the Riemannian distance $\dist$ on $(M,g)$.}
Then, {by employing balls with radii going to infinity, see for instance \cite[Corollary 3.6]{AntBruFogPoz},} it is not difficult to check that also the converse inequality holds, so that
\begin{equation}
    c_{\mathrm{iso}}(M,g)= \left(\mathrm{AVR}(M,g)\right)^{\frac{1}{N-1}}\, .
\end{equation}
This shows that the large scale geometry of a manifold with non negative Ricci curvature influences its isoperimetric behaviour down to the bottom \emph{volume scale}.\\
Thanks to \autoref{thm:BavardPansuIntro}, we understand that this is true for any two intermediate volume scales, namely the {scale invariant isoperimetric profile} is monotone decreasing with respect to the volume.  

Notice that, in terms of the isoperimetric profile function $I$, the isoperimetric cone angle can be equivalently characterized as
\begin{equation}
c_{\mathrm{iso}}(M,g)=\inf\left\{\frac{I(v)^{\frac{N}{N-1}}}{\left(N\omega_{N}^{\frac{1}{N}}\right)^{\frac{N}{N-1}}v}\, :\, v\in (0,\infty)\right\}\, ,
\end{equation}
or, introducing the {scale invariant isoperimetric profile} at volume $v\in(0,\haus^N)$ as
\begin{equation*}
    c_{\mathrm{iso}}(M,g)(v):=\frac{I(v)^{\frac{N}{N-1}}}{\left(N\omega_{N}^{\frac{1}{N}}\right)^{\frac{N}{N-1}}v}=\inf\left\{\frac{\left(\haus^{N-1}(\partial \Omega)\right)^{\frac{N}{N-1}}}{\left(N\omega_{N}\right)^{\frac{1}{N-1}}\haus^N(\Omega)}\, :\, \emptyset \neq \Omega\subset M\, , \haus^N(\Omega)=v\right\}\, , 
\end{equation*}
by 
\begin{equation*}
c_{\mathrm{iso}}(M,g)=\inf_{v\in (0,\infty)}c_{\mathrm{iso}}(M,g)(v)\, .
\end{equation*}
Below we prove that, in the greater generality of $\RCD(0,N)$ spaces $(X,\dist,\haus^N)$, the {scale invariant isoperimetric profile} is monotone decreasing with respect to the volume, without any further assumption on the volume growth. {Analogous statements for smooth Riemannian manifolds with nonnegative Ricci curvature that are either compact or have uniformly controlled geometry at infinity were obtained in \cite{BavardPansu86,Bayle03,MondinoNardulli16}.}
%Therefore, the isoperimetric behaviour for a given volume controls the isoperimetric behaviour for any smaller volume. 

\begin{theorem}\label{cor:IsoperimetricProfileRCD0N}
Let $(X,\dist,\haus^N)$ be an $\RCD(0,N)$ space with isoperimetric profile function $I$. The following hold:
\begin{enumerate}
    \item the function $I^{\frac{N}{N-1}}$ is concave on $(0,\haus^N(X))$. A fortiori $I$ is concave on $(0,\haus^N(X))$ and, if $\haus^N(X)=+\infty$, $I$ is nondecreasing on $(0,+\infty)$;
    \item we have that
    \begin{equation}\label{eq:monotonicityformula}
    v\mapsto \frac{I(v)}{v^{\frac{N-1}{N}}}\,\,\text{is non-increasing on $(0,\haus^N(X))$}
    \end{equation}
    and, when $\haus^N(X)=+\infty$,
    \begin{equation}\label{eq:aslargev}
    \lim_{v\to \infty}\frac{I(v)}{v^{\frac{N-1}{N}}} = N\left(\omega_N\mathrm{AVR}(X,\dist,\haus^N)\right)^{\frac 1N}\, ;
    \end{equation}
    \item when $\haus^N(X)=+\infty$,
    \begin{equation}
    \lim_{v\to \infty} v^{\frac{1}{N}}I_+'(v)=(N-1)\left(\omega_N\mathrm{AVR}(X,\dist,\haus^N)\right)^{\frac 1N}\, ,
    \end{equation}
    where $I_+'(v)$ is the right derivative of $I$;
    \item if $\mathrm{AVR}(X,\dist,\haus^N)>0$, then $I$ is strictly increasing and strictly concave. 
\end{enumerate}
\end{theorem}

\begin{proof}
Let us notice that if $\inf_x\haus^N(B_1(x))=0$, we have $I\equiv 0$, see \cite[Proposition 2.18]{AntBruFogPoz}, whose proof adapts in the non-smooth setting since it only relies on Bishop--Gromov monotonicity. Thus this case is trivial. Let us suppose, from now on, that $\inf_x\haus^N(B_1(x))>0$.

Item (1) readily follows from the continuity of $I$, see \autoref{rem:LocalHolderProfile}, and \eqref{eqn:Bayle}.

Let us prove item (2). Since $I^{\frac{N}{N-1}}$ is concave and $\lim_{v\to 0^+}I^{\frac{N}{N-1}}(v)=0$, we get from concavity that
$$
v\mapsto \frac{I^{\frac{N}{N-1}}(v)}{v},
$$
is non-increasing. Hence we get the first part of the assertion. The asymptotic \eqref{eq:aslargev} follows as in \cite[Corollary 3.6]{AntBruFogPoz}, since it only relies on the Bishop--Gromov monotonicity and the isoperimetric inequality in \autoref{thm:IsoperimetricSharpRigidintro}.

Item (3) follows verbatim as in \cite[Corollary 3.6]{AntBruFogPoz}, since it only relies on the Bishop--Gromov monotonicity and the concavity of $I$.

Let us prove item (4). The fact that $I$ is strictly increasing follows verbatim as in \cite[Corollary 3.8]{AntBruFogPoz} by exploiting the previous items. The fact that $I$ is strictly concave follows from the fact that $I^{\frac{N}{N-1}}$ is concave and strictly increasing.
\end{proof}

The monotonicity of the isoperimetric profile in \autoref{cor:IsoperimetricProfileRCD0N} also directly implies the following consequence.

\begin{corollary}
Let $(X,\dist,\haus^N)$ be an $\RCD(0,N)$ space with infinite volume.

If $E\subset X$ is an isoperimetric region, then $\Per(E)\le \Per(F)$ whenever $\haus^N(E)\le\haus^N(F)$. In particular $E$ is outward minimizing, i.e., $\Per(E)\le \Per(F)$ whenever $E\subset F$.

If also ${\rm AVR}(X,\dist,\haus^N)>0$, then the previous inequalities are strict and isoperimetric sets are strictly outward minimizing.
\end{corollary}

% \begin{proof}
% The result is a direct corollary of the monotonicity of the isoperimetric profile in \autoref{cor:IsoperimetricProfileRCD0N}.
% \end{proof}

\subsection{Consequences for other geometric and functional inequalities}\label{subsec:functineq}

Here we follow a classical strategy to obtain (sharp) functional inequalities from (sharp) isoperimetric inequalities arguing by rearrangement, see for instance \cite{PolyaSzego,BerardMeyer82,Maziabooksobolev,Grigorianisomazia}.\\ 
Since we are able to characterize the rigidity in the sharp isoperimetric inequalities, the characterization of rigidity in the sharp functional inequalities will follow as well.\\
We will improve the existing results in several directions:
\begin{itemize}
    \item taking into account the new monotonicity of the quotient $v\mapsto I(v)/v^{\frac{n-1}{n}}$, our statements will depend on the isoperimetric behaviour of the space on a fixed range of volumes $v\in [0,\bar{v}]$, rather than on the full range (or, equivalently, on the asymptotic volume ratio);
    \item we will characterize the rigidity without technical regularity assumptions for smooth Riemannian manifolds, improving upon the recent \cite{BaloghKristaly};
    \item we will characterize the rigidity in the more general setting of $\RCD(0,N)$ spaces $(X,\dist,\haus^N)$. This setting includes as remarkable examples Alexandrov spaces with non negative sectional curvature and cones with non negative Ricci curvature.
\end{itemize}

We will focus on $\RCD(0,N)$ metric measure spaces $(X,\dist,\haus^N)$.

Let us borrow the terminology from \cite{MondinoSemolaPolyaSzego} (see also \cite[Section 3]{NobiliViolo}). Given a metric measure space $(X,\dist,\meas)$, an open set $\Omega\subset X$ with $\meas(\Omega)<\infty$ and a Borel function $u:\Omega\to[0,\infty)$ we will denote by
\begin{equation}
\mu(t):=\meas\left(\{u>t\}\right)\, ,
\end{equation}
the distribution function of $u$ and by $u^{\sharp}$ the generalized inverse function of $\mu$.\\
Moreover, given $N\ge 1$ we choose $0<r<\infty$ such that $\meas_N([0,r])=\meas(\Omega)$, where $\meas_N=N\omega_N r^{N-1}\di r$ and define the monotone rearrangement $u^*$ of $u$ by
\begin{equation}
u^*(x):=u^{\sharp}\left(\meas_N([0,x])\right)\, ,\quad\text{for any $x\in [0,r]$}\, .    
\end{equation}
Notice that, by the very definition, $u$ and $u^*$ have the same distribution function, which implies that 
\begin{equation}
\int_{\Omega} f(u)\di\meas=\int _{[0,r]}f(u^*)\di\meas_N\, ,\quad\text{for any Borel function $f:[0,\infty)\to[0,\infty)$}\, .
\end{equation}
Let $1<p<+\infty$. For any non negative function $u\in W^{1,p}_0(\Omega)$. Let us also introduce a function $f_u:[0,\sup u^*]\to [0,\infty)$ by
\begin{equation}
f_u(t):=\int \abs{\nabla u^*}^{p-1}\di\Per(\{u^*>t\})\, ,
\end{equation}
and notice that, by the coarea formula,
\begin{equation}
\int_0^{\sup u^*}f_u(t)\di t=\int _{[0,r]}\abs{\nabla u^*}^p\di\meas_N\, .
\end{equation}

We shall denote by $I_N$ the isoperimetric profile of $\setR^N$ with the canonical Euclidean structure. Notice that it coincides with the isoperimetric profile of the model one dimensional metric measure space $\left([0,\infty),|\cdot |,N\omega_Nr^{N-1}\di r\right)$ and it holds $I_N(v)=N\omega_N^{1/N}v^{\frac{N-1}{N}}$ for any $v\ge 0$.\\ 
The classical rearrangement argument (see \cite{BerardMeyer82,PolyaSzego} for the classical formulations and \cite{MondinoSemolaPolyaSzego,NobiliViolo} for a more recent one in the setting of $\CD(K,N)$ spaces) gives the following.

\begin{proposition}\label{prop:strongPolyaSzego}
Let $(X,\dist,\haus^N)$ be an $\RCD(0,N)$ metric measure space. Let $1<p<+\infty$. Let $\Omega\subset X$ be an open domain with $\haus^N(\Omega)<\infty$. Let $u\in W^{1,p}_0(\Omega)$ be non negative and let $u^*:[0,r]\to [0,\infty)$ be its monotone rearrangement, where $r>0$ is such that $\meas_N([0,r])=\haus^N(\Omega)$. Then
\begin{equation}
\int _{\Omega}\abs{\nabla u}^p\di\haus^N\ge \int_0^{\sup u^*}\left(\frac{I_{(X,\dist,\haus^N)}(\mu(t))}{I_N(\mu(t))}\right)^pf_u(t)\di t\, .
\end{equation}
\end{proposition}

We omit the proof that can be obtained arguing as in \cite{MondinoSemolaPolyaSzego,NobiliViolo}. We just point out, since this will be relevant in order to address the rigidity issue, that the stronger statement
\begin{equation}\label{eq:imppssuper}
 \int _{\Omega}\abs{\nabla u}^p\di\haus^N\ge \int_0^{\sup u^*}\left(\frac{\Per(\{u>t\})}{I_N(\mu(t))}\right)^pf_u(t)\di t\,    
\end{equation}
can be obtained for any non negative function $u\in W^{1,p}_0(\Omega)$, for $1<p<+\infty$, such that $u^*$ has $\leb^1$-almost everywhere non vanishing derivative on $(0,r)$, see \cite[Proposition 3.13]{MondinoSemolaPolyaSzego}.

\begin{corollary}\label{cor:improvedPS}
Let $(X,\dist,\haus^N)$ be an $\RCD(0,N)$ metric measure space. Let $1<p<+\infty$. Let $\Omega\subset X$ be an open domain with $\haus^N(\Omega)<\infty$. Let $u\in W^{1,p}_0(\Omega)$ be non negative and let $u^*:[0,r]\to [0,\infty)$ be its monotone rearrangement, where $r>0$ is such that $\meas_N([0,r])=\haus^N(\Omega)$. Then
\begin{equation}\label{eq:improvedPS}
\int _{\Omega}\abs{\nabla u}^p\di\haus^N\ge \left(\frac{I_{(X,\dist,\haus^N)}(\haus^N(\Omega))}{I_N(\haus^N(\Omega))}\right)^p\int_{[0,r]}\abs{\nabla u^*}^p\di\meas_N\, .
\end{equation}
\end{corollary}

\begin{proof}
The result follows from \autoref{prop:strongPolyaSzego}, thanks to the monotonicity formula \eqref{eq:monotonicityformula}.
\end{proof}

\begin{remark}
Let $1<p<+\infty$. In \cite{AgostinianiFogagnoloMazzieri,FogagnoloMazzieri,BaloghKristaly} under different assumptions the comparison 
\begin{equation}\label{eq:classPS}
 \int _{\Omega}\abs{\nabla u}^p\di\haus^N\ge \left(\mathrm{AVR}(X,\dist,\haus^N)\right)^\frac{p}{N}\int_{[0,r]}\abs{\nabla u^*}^p\di\meas_N\,   
\end{equation}
was deduced from the sharp isoperimetric inequality, via symmetric rearrangement.\\
The estimate \eqref{eq:improvedPS} is easily seen to imply \eqref{eq:classPS} since
\begin{equation}
v\mapsto \frac{I_{(X,\dist,\haus^N)}(v)}{I_N(v)}\downarrow \left(\mathrm{AVR}(X,\dist,\haus^N)\right)^{\frac{1}{N}}\, ,\quad\text{as $v\to +\infty$}\, ,
\end{equation}
by \autoref{cor:IsoperimetricProfileRCD0N}. In fact, it is strictly stronger, since it requires control over the isoperimetric profile up to volume $\haus^N(\Omega)$ instead of requiring control over the isoperimetric profile for all volumes (or, equivalently, on the asymptotic volume ratio of $(X,\dist,\haus^N)$).
\end{remark}

Given any $1<p<\infty$ and any bounded and open domain $\Omega\subset X$ such that $\haus^N(X\setminus \Omega)>0$ we shall denote by $\lambda^{1,p}(\Omega)$ the first Dirichlet eigenvalue of the $p$-Laplacian with homogeneous Dirichlet boundary conditions on $\Omega$. It is well known that it admits the classical variational interpretation
\begin{equation}
\lambda^{1,p}(\Omega):=\inf\left\{\frac{\int_{\Omega} \abs{\nabla f}^p\di\haus^N}{\int_{\Omega}f^p\di\haus^N}\, :\,  f\in \Lip_c(\Omega)\right\}\, .
\end{equation}
Moreover, following \cite{Coulhon95dimension,Coulhon03iso} we introduce a whole scale of $p$-isoperimetric profile functions $I_p:[0,\infty)\to[0,\infty)$ by
\begin{equation}
I_p(v):=\inf\left\{\lambda^{1,p}(\Omega)\, : \, \Omega\subset X\, ,\haus^N(\Omega)=v\right\}\, .
\end{equation}
We shall also denote by $I_{p,N}:(0,\infty)\to(0,\infty)$ the function associating to any volume $v\in (0,\infty)$ the lowest first eigenvalue of the $p$-Laplacian with Dirichlet boundary conditions on open domains $\Omega\subset \setR^N$ with $\leb^N(\Omega)=v$. It is well known that $I_{p,N}$ is the first eigenvalue of the $p$-Laplacian with Dirichlet boundary conditions on a ball $B_r(0^N)$ such that $\leb^N(B_r(0^N))=v$. In particular, $I_{p,N}(v)=C_{p,N}v^{-\frac{p}{N}}$, for some constant $C_{p,N}>0$.

\begin{theorem}
Let $(X,\dist,\haus^N)$ be an $\RCD(0,N)$ metric measure space. Then, for any $1<p<\infty$ it holds
\begin{equation}\label{eq:piso}
\frac{I_p(v)}{I_{p,N}(v)}\ge \left(\frac{I_{(X,\dist,\haus^N)}(v)}{I_N(v)}\right)^p\ge \left(\mathrm{AVR}(X,\dist,\haus^N)\right)^{\frac{p}{N}} \, ,\quad\text{for any $0<v<\infty$}\, .
\end{equation}

\end{theorem}

\begin{proof}
The statement follows directly from \autoref{prop:strongPolyaSzego} and \autoref{cor:improvedPS} by the classical symmetric rearrangement argument and the variational characterization of the first eigenvalue of the $p$-Laplacian with Dirichlet boundary conditions.
\end{proof}

\begin{remark}
The fact that the isoperimetric profile controls the whole scale of $p$-isoperimetric profiles is classical \cite{Coulhon95dimension,Grigoryanbook}. However, in all the references we are aware of, control was intended up to constants. The observation that the isoperimetric profile controls the $p$-spectral gap without the need of additional constants seems to be new even for smooth Riemannian manifolds with non negative Ricci curvature and it fully exploits the monotonicity \autoref{cor:IsoperimetricProfileRCD0N}.\\
\end{remark}

\begin{corollary}\label{cor:rigidityspectral}
Let $(X,\dist,\haus^N)$ be an $\RCD(0,N)$ metric measure space with $\mathrm{AVR}(X,\dist,\haus^N)>0$. Let $1<p<+\infty$. Let $\Omega\subset X$ be an open and bounded domain. If 
\begin{equation}
\lambda^{1,p}(\Omega)=\left(\mathrm{AVR}(X,\dist,\haus^N)\right)^{\frac{p}{N}}I_{p,N}(\haus^N(\Omega))\, ,
\end{equation}
then $(X,\dist,\haus^N)$ is isomorphic to a metric measure cone over an $\RCD(N-1,N)$ metric measure space $(Y,\dist_Y,\haus^{N-1})$. In particular, if $(X,\dist)$ is isometric to a smooth Riemannian manifold, then $(X,\dist)$ is isometric to $\setR^N$.
\end{corollary}

\begin{proof}
The proof is similar to the analogous rigidity statement proved for $\RCD(N-1,N)$ metric measure spaces in \cite{MondinoSemolaPolyaSzego} building on the top of the rigidity statement for the L\'evy--Gromov isoperimetric inequality obtained in \cite{CM17}, so we just outline it.

Notice that the spectral gap of $\Omega$ is attained by a function $u\in W^{1,p}_0(\Omega)$, by a classical argument. Under our assumptions, the rearrangement $u^*$ is a first eigenfunction of the $p$-Laplacian with Dirichlet boundary conditions on $([0,r],|\cdot|,\meas_N)$. In particular, it is a classical fact that it has non vanishing derivative $\leb^1$-a.e. on $[0,r]$ and \eqref{eq:imppssuper} holds.\\
Since equality holds in the spectral gap inequality, equality holds in \eqref{eq:imppssuper}. In particular,
\begin{equation}
\Per(\{u>t\})=\mathrm{AVR}(X,\dist,\haus^N)^{\frac{1}{N}}I_N(\haus^N(\{u>t\}))\, ,
\end{equation}
for $\leb^1$-a.e. $t\in [0,\sup u^*]$.
This is sufficient to get the first conclusion in the statement, thanks to \autoref{lem:RigidityIsop}.\\
The second conclusion directly follows since the unique $\RCD(0,N)$ metric measure cone $(X,\dist,\haus^N)$ which is smooth is $\setR^N$. 
\end{proof}

\begin{remark}
The rigidity \autoref{cor:rigidityspectral} is new also for smooth Riemannian manifolds with non negative Ricci curvature and Euclidean volume growth, since it removes the additional regularity assumptions required in \cite{BaloghKristaly}.
\end{remark}

\subsection{Asymptotic isoperimetric behaviour for nonnegatively Ricci curved spaces with stable asymptotic cones and Euclidean volume growth}

In the context of $\RCD(0,N)$ spaces verifying suitable conditions on their asymptotic cones and with Euclidean volume growth we can show, heavily leveraging on the ideas developed in \cite{AntBruFogPoz}, that isoperimetric regions exist for any sufficiently large volume, and that, up to translations along Euclidean factors and scalings, they converge to balls in the asymptotic cone at infinity in the Hausdorff sense.\\
Moreover, in the same class, the rigidity in \autoref{thm:IsoperimetricSharpRigidintro} holds just under the assumption that the isoperimetric profile of the space equals the one of the cone with same ${\rm AVR}$ and dimension, \emph{for some} volume $V>0$. We stress again that the latter class of non negatively Ricci curved spaces encompasses the class of finite dimensional Alexandrov spaces with nonnegative curvature and Euclidean volume growth, cf. \autoref{rem:AncheAlexandrov}. 

\begin{proof}[Proof of \autoref{thm:AlexandrovRigidIntro}]
The first item comes from a straightforward adaptation of the proof of \cite[Theorem 1.2]{AntBruFogPoz} to the setting of $\RCD(0,N)$ spaces, taking into account the generalized asymptotic mass decomposition \autoref{thm:MassDecompositionINTRO} and the properties proved in \autoref{cor:IsoperimetricProfileRCD0N}.\\ 
Let us sketch the proof, by referring the reader to \cite{AntBruFogPoz} for the complete argument. The proof of \cite[Theorem 1.2]{AntBruFogPoz} leverages on \cite[Lemma 4.2]{AntBruFogPoz}, which is already formulated in the non smooth setting, and on \cite[Theorem 1.1]{AntBruFogPoz}. The analogue of \cite[Theorem 1.1]{AntBruFogPoz} can be readily proven in the setting of $\RCD(0,N)$ spaces. Indeed, as a consequence of item (3) of \autoref{cor:FinePropertiesProfile}, and item (1) of \autoref{lem:IsoperimetricAtFiniteOrInfinite}, the minimization process on the space, when ran as in \autoref{thm:MassDecompositionINTRO}, either produces an isoperimetric region or exactly one piece escaping in a pmGH limit at infinity. Having this at disposal the contradiction argument of \cite[Theorem 1.1]{AntBruFogPoz} applies verbatim also in the non smooth case.
\medskip

In order to prove the second item, assume first that $X$ does not split any line. Hence, from the assumption, no asymptotic cone of it splits any line as well.
%(see \cite[pages 58-59]{BallmannGromovSchroeder85}, \cite[Proposition 4.2]{Kasue88} or \cite[Theorem 4.6, Remark 4.7]{AntBruFogPoz} for a detailed proof in the case of positive ${\rm AVR}$).

Let us fix $o\in X$ and denote $(X_i,\dist_i,\haus^N_{\dist_i},o):=(X,V_i^{-1/N}\dist,V_i^{-1}\haus^N,o)$. Let us also take $q_i\in E_i$. Taking into account \eqref{eq:BoundDiamBig}, for every $i\geq 1$, the following hold
\begin{equation}\label{eqn:BoundVolDiam}
\haus^N_{\dist_i}(E_i)=1, \qquad \diam_{\dist_i}(E_i) \leq D\, ,
\end{equation}
where $D$ is a constant depending on $N,\mathrm{AVR}(X,\dist,\haus^N)$.
Moreover, $(X_i,\dist_i,\haus^N_{\dist_i},o)\to (C,\dist_\infty,\haus^N,\tilde o)$, as $i\to +\infty$, where $(C,\dist_\infty,\haus^N,\tilde o)$ is an asymptotic cone of $(X,\dist,\haus^N)$. 

We now distinguish two cases.
\begin{itemize}
    \item Suppose
    $\dist_i(o,q_i)=V_i^{-1/N}\dist(o,q_i)\to +\infty$. Hence $\dist(o,q_i)\to +\infty$ as $i\to +\infty$. Let us call $r_i:=\dist_i(o,q_i)$. Let $(C',\dist_\infty',\haus^N,\tilde o')$ be an asymptotic cone obtained as a limit of a subsequence of $(X,r_i^{-1}\dist, r_i^{-N}\haus^N,o)$. From \cite[Lemma 4.2]{AntBruFogPoz} and the fact that $C'$ does not split a line by assumption, we have that $\vartheta[C',\dist_\infty',\haus^N,o']\geq \mathrm{AVR}(X,\dist,\haus^N)+\varepsilon$, for some $\varepsilon>0$, and for every $o'\in C'$ such that $\dist_\infty(\tilde o',o')=1$. Arguing as in the proof of \cite[Lemma 4.2]{AntBruFogPoz}, by volume convergence there exists $\rho>0$ such that 
    \[
    \frac{\haus^N(B_{\rho r_i}(q_i))}{\omega_N(\rho r_i)^N}\geq \mathrm{AVR}(X,\dist,\haus^N)+\varepsilon/2\, ,
    \]
    for every $i$ sufficiently large. 
    
    Let $(X_i,\dist_i,\haus^N_{\dist_i},q_i)$ converge, up to subsequence, to a pmGH limit $(C'',\dist'',\haus^N,q_\infty)$.
    By Bishop--Gromov monotonicity, for every $R>0$ and for any $i$ sufficiently large the following holds
    \[
    \frac{\haus^N(B_{RV_i^{-1/N}}(q_i))}{\omega_N\left(RV_i^{-1/N}\right)^N}\geq \frac{\haus^N(B_{\rho r_i}(q_i))}{\omega_N\left(\rho r_i\right)^N}\, ,
    \]
    since $\dist(o,q_i)\to +\infty$. Hence, by volume convergence, 
    \[
    \frac{\haus^N(B_R(q_\infty))}{\omega_NR^N}\geq \mathrm{AVR}(X,\dist,\haus^N)+\varepsilon/2\, ,
    \]
    for every $R>0$, from which $\mathrm{AVR}(C'',\dist'',\haus^N)\geq \mathrm{AVR}(X,\dist,\haus^N)+\varepsilon/2$.
    
    By \eqref{eqn:BoundVolDiam}, we can apply \cite[Proposition 3.3]{AmbrosioBrueSemola19}. Therefore, up to subsequences, $E_i\subset (X_i,\dist_i,\haus^N_{\dist_i},q_i)$ converges in $L^1$-strong to a set of finite perimeter $E_\infty\subset (C'',\dist'',\haus^N,q_\infty)$ with $\haus^N(E_\infty)=1$. Moreover
    \begin{equation}\label{eqn:CONTRA}
    \begin{split}
        \Per(E_\infty)&\leq \liminf_{i\to +\infty} \Per (E_i) = \liminf_{i\to +\infty} V_i^{-\frac{N-1}{N}} \Per(E_i) \\ &=\liminf_{i\to +\infty} V_i^{-\frac{N-1}{N}} I(V_i) = N(\omega_N\mathrm{AVR}(X,\dist,\haus^N))^{\frac{1}{N}}\, ,
    \end{split}
    \end{equation}
    where in the last equality we are using the asymptotic in \autoref{cor:IsoperimetricProfileRCD0N}. However, from the sharp isoperimetric inequality on $C''$ and taking into account that $\haus^N(E_\infty)=1$, we have
    \begin{equation*}
        \Per(E_\infty)\geq N(\omega_N\mathrm{AVR}(C'',\dist'',\haus^N))^{\frac{1}{N}}\geq N(\omega_N(\mathrm{AVR}(X,\dist,\haus^N)+\varepsilon/2))^{\frac{1}{N}}\, ,
    \end{equation*}
    which is a contradiction with \eqref{eqn:CONTRA}. Thus this case cannot occur.
 
    \item Hence $\sup_{i\in\mathbb N}\dist_i(o,q_i)<+\infty$. In this case, up to subsequence, there exists a point $o'$ at finite distance from $o$ in $C$ such that $(X_i,\dist_i,\haus^N_{\dist_i},q_i)\to (C,\dist_\infty,\haus^N,o')$. Arguing as in the previous case we get that, up to subsequences, by the lower semicontinuity of the perimeter and the sharp isoperimetric inequality, the sequence $E_i$ converges in $L^1$-strong to $E\subset (C,\dist_\infty,\haus^N)$ such that $\haus^N(E)=1$ and
    \[
    \Per(E)=N(\omega_N\mathrm{AVR}(X,\dist,\haus^N))^{\frac{1}{N}}\, .
    \]
    Hence, by the rigidity of the isoperimetric inequality, $E$ is isometric to the ball of volume $1$ centered at some tip in $C$.
\end{itemize}
By \autoref{thm:ConvergenceBarriersStability}, convergence in $L^1$-strong of the $E_i$'s implies convergence in Hausdorff distance in some realization. Thus the proof is completed when $X$ does not split any line. \medskip

In the remaining general case, we can write $X=\mathbb R^k\times \tilde X$, with $k\geq 1$, and such that $\tilde X$ is such that no asymptotic cone of $\tilde X$ splits a line. Let $o$ and $q_i$ be as above. Up to a translation along the Euclidean factor, we may always assume that $q_i=(0,\tilde q_i)$, where $\tilde q_i\in\tilde X$. This prevents the fact that the component along $\mathbb R^k$ of $q_i$ might go to infinity.
Hence one can argue as in the previous case, by slightly modifying the argument, distinguishing the case in which $\dist_i(o,q_i)\to +\infty$, which eventually does not occur, and the case $\sup_i \dist_i(o,q_i)<+\infty$.
\medskip

Let us now prove the third item. Let $\Omega_i\subset X$ be a perimeter minimizing sequence of measure $\haus^N(\Omega_i)=V$, i.e., $\Per(\Omega_i)\to I(V)$, and let $p_{i,j}, p_j, \Omega, Z_j, X_j, \overline{N}$ be given by \autoref{thm:MassDecompositionINTRO}. If no mass is lost in the limit, then there exists an isoperimetric region of volume $V$ and the result follows from \autoref{thm:IsoperimetricSharpRigidintro}.\\
%together with the characterization of cones which are Alexandrov spaces, see \cite[Theorem 4.7.1]{BuragoBuragoIvanovBook}.\\ 
So let us assume that $\overline{N}\ge1$, and then, by \autoref{lem:IsoperimetricAtFiniteOrInfinite}, we actually have $\overline{N}=1$ and $\haus^N(\Omega)=0$. We have that $X=\setR^{k}\times \tilde X$ for $k \in\{0,\ldots,N\}$ and $\dist=\dist_{\rm eu}\times\dist_{\tilde X}$, where $\tilde X$ is such that no asymptotic cone of it splits a line.
%an Alexandrov space of dimension $N-k$ with non negative curvature which does not split lines.
We can rename the only diverging sequence of points $p_{i,1}$ into $p_{i,1}=(q_{i},y_{i})$. It is a standard fact to check that ${\rm AVR}(\tilde X,\dist_{\tilde X},\haus^{N-k})={\rm AVR}(X,\dist,\haus^N)$. Observe that in this case, $X_1$ is the only limit space obtained by applying \autoref{thm:MassDecompositionINTRO}.

If $\sup_i \dist_{\tilde X}(y_{i},y_0)<+\infty$ for some $y_0\in \tilde X$, then $X_1$ is isometric to $X$. Hence the limit set $Z_1$ is an isoperimetric region of volume $V$ in a copy of $X$, and then again the claim follows as if no mass were lost.

Assume then that $\lim_i \dist_{\tilde X}(y_{i},y_0)=+\infty$.
We can write $X_1=\setR^{k}\times \widetilde{Y}$ and $p_1=(q,y)$, where $(\tilde X,\dist_{\tilde X},\haus^{N-k},y_{i})\to (\widetilde{Y}, \dist_{\widetilde{Y}}, \haus^{N-k}, y)$ in the pmGH sense. Recall that no asymptotic cone of $\tilde X$ splits a line.
%, its unique asymptotic cone does not split lines either, since the proofs of \cite[Theorem 4.6, Remark 4.7]{AntBruFogPoz} verbaim hold also in the non smoot setting. 
It then follows from \cite[Lemma 4.2]{AntBruFogPoz} that there exists $\eps>0$ such that ${\rm AVR}(X_1,\dist_{X_1},\haus^N)={\rm AVR}(\widetilde{Y}, \dist_{\widetilde{Y}}, \haus^{N-k})\ge {\rm AVR}(\tilde X,\dist_{\tilde X},\haus^{N-k}) + \eps $. Hence, by \autoref{thm:MassDecompositionINTRO} and \autoref{thm:IsoperimetricSharpRigidintro}, we get that
\begin{equation}
    \begin{split}
        N\omega_N^{\frac1N}&({\rm AVR}(X,\dist,\haus^N))^{\frac1N}V^{\frac{N-1}{N}}
        = I(V) =  \Per (Z_1) \\&\ge N\omega_N^{\frac{1}{N}}
        \left(\mathrm{AVR}(X,\dist,\haus^N)+\eps\right)^{\frac{1}{N}} V^{\frac{N-1}{N}},
    \end{split}
\end{equation}
which yields a contradiction.
\medskip
\end{proof}

\begin{remark}\label{rem:AncheAlexandrov}
Even though the existence above is shown only for big volumes, the rigidity \autoref{thm:AlexandrovRigidIntro} holds whenever equality \eqref{eq:RigidProfileAlexandrov} is achieved at some volume $V>0$.\\
Moreover, notice that every Alexandrov space of dimension $N$ with nonnegative curvature and Euclidean volume growth falls in the hypotheses of the first part of \autoref{thm:AlexandrovRigidIntro}, since \cite[Theorem 4.6]{AntBruFogPoz} holds with the same proof in the setting of Alexandrov spaces.
Hence, \autoref{thm:AlexandrovRigidIntro}, when specialized to the Alexandrov setting, gives raise to a complete generalization of \cite[Theorem 6.3]{LeonardiRitore} and \cite[Theorem 6.14]{LeonardiRitore} to the setting of non negatively curved Alexandrov spaces with Euclidean volume growth, which is a class that strictly contains the one of convex bodies of $\mathbb R^N$ with non-degenerate asymptotic cone considered in \cite[Section 6]{LeonardiRitore}.
\end{remark}

\section{Asymptotic isoperimetric behaviour for small volumes and almost regularity theorems}

In this last section we combine the second order differential inequalities and monotonicity for the isoperimetric profile with the sharp and rigid isoperimetric inequality for $\RCD(0,N)$ spaces $(X,\dist,\haus^N)$ and the strong stability of isoperimetric regions to determine the asymptotic isoperimetric behaviour for small volumes of $\RCD(K,N)$ spaces with a uniform lower bound on the volume of unit balls. Then we prove new global $\eps$-regularity results formulated in terms of the isoperimetric profile.

\subsection{Asymptotic isoperimetric behaviour for small volumes}

%For the forthcoming \autoref{cor:AsymptoticProfileZero}, that describes the asymptotic of the isoperimetric profile at zero, we shall need the following lemma.\\
Given an $\RCD(K,N)$ space $(X,\dist,\meas)$, we call $\vartheta[X,\dist,\meas,x]$ the \emph{density at the point $x\in X$} defined as
\[
\vartheta[X,\dist,\meas,x]:=\lim_{r\to 0}\frac{\meas(B_r(x))}{\omega_Nr^N}=\lim_{r\to 0}\frac{\meas(B_r(x))}{v(N,K/(N-1),r)}.
\]

\begin{lemma}\label{lem:InequalitiesIsopProfile}
Let $(X,\dist,\haus^N)$ be an $\RCD(K,N)$ space such that $\haus^N(B_1(x))\geq v_0>0$ for every $x\in X$ and some $v_0>0$. Then the following hold.
\begin{enumerate}
    \item For every $p\in X$ we have that 
    \[
    \lim_{v\to 0}\frac{I_X(v)}{v^{\frac{N-1}{N}}}\leq N(\omega_N\vartheta[X,\dist,\haus^N,p])^{1/N}.
    \]
    As a consequence, if $K=0$, for every tangent cone $C_p$ at $p$ we have 
    \[
    I_X(v)\leq I_{C_p}(v), \qquad \text{for every $v>0$}.
    \]

    \item Let $\{p_i\}$ be a sequence of points on $X$. Up to subsequences $(X,\dist,\haus^N,p_i)$ converge to an $\RCD(K,N)$ space $(X_\infty,\dist_\infty,\haus^N,p_\infty)$. Let $\lambda_i\to +\infty$ be a diverging sequence. Up to subsequences, $(X,\lambda_i\dist,\haus^N_{\lambda_i\dist},p_i)$ converge to an $\RCD(0,N)$ space $(X',\dist',\haus^N,p')$. Then 
    \[
    \mathrm{AVR}(X',\dist',\haus^N)\geq \vartheta[X_\infty,\dist_\infty,\haus^N, p_\infty].
    \] 
    As a consequence, if $C_\infty$ is a tangent cone of $X_\infty$ at $p_\infty$, the following holds
    \[
    I_{C_\infty}(v)\leq I_{X'}(v), \qquad \text{for every $v>0$}.
    \]
\end{enumerate}
\end{lemma}

\begin{proof}
Let us prove the first item. Let us denote for simplicity $\vartheta:=\vartheta[X,\dist,\haus^N,p]>0$. By definition of density and thanks to the Bishop--Gromov monotonicity and the coarea formula, we get that 
\[
\vartheta=\lim_{r\to 0} \frac{\haus^N(B_r(p))}{\omega_Nr^N}=\mathrm{esslim}_{r\to 0} \frac{\Per(B_r(p))}{N\omega_Nr^{N-1}}\, .
\]

Hence, for every $\varepsilon>0$ there exists $v_\varepsilon$ such that for every $v\leq v_\varepsilon$ there exists $r_v$ with $\haus^N(B_{r_v}(p))=v$ and 
\[
\Per(B_{r_v}(p))\leq N(\omega_N(\vartheta+\varepsilon))^{1/N}v^{\frac{N-1}{N}}\, .
\]
Then, for every $\varepsilon>0$ there exists $v_\varepsilon>0$ such that 
\[
\frac{I(v)}{v^{\frac{N-1}{N}}}\leq N(\omega_N(\vartheta+\varepsilon))^{1/N}, \qquad \text{for all $v\leq v_\varepsilon$}\, .
\]
Hence for every $\varepsilon>0$ we have that, taking into account the existence of the limit in item (1) in \autoref{cor:FinePropertiesProfile}, 
\begin{equation}\label{eq:u2}
\lim_{v\to 0}\frac{I(v)}{v^{\frac{N-1}{N}}}\leq N(\omega_N(\vartheta+\varepsilon))^{1/N}\, .
\end{equation}
Thus taking $\varepsilon\to 0$ in \eqref{eq:u2}
we get the first sought conclusion of the first item. 

The second conclusion of the first item follows from the first, the fact that $v\mapsto I(v)/v^{\frac{N-1}{N}}$ is nonincreasing thanks to item (2) of \autoref{cor:IsoperimetricProfileRCD0N}, and \autoref{cor:IsoperimetricOnCones}, since the opening at any tip of a tangent cone at $p$ is $\vartheta[X,\dist,\haus^N,p]$.
\medskip

Let us prove the second item. Let us call for simplicity $\vartheta:=\vartheta[X_\infty,\dist_\infty,\haus^N,p_\infty]$. For every $\varepsilon>0$ there exists $r_\varepsilon>0$ such that
\begin{equation}\label{eq:carg}
\frac{\haus^N(B_r(p_\infty))}{v(N,K/(N-1),r)}>\vartheta-\varepsilon, \qquad \text{for all $0<r\leq r_\varepsilon$}\, .
\end{equation}
By volume convergence and \eqref{eq:carg}, for $i$ large enough we have 
\begin{equation}
\haus^N(B_{r_\varepsilon/2}(p_i))/(v(N,K/(N-1),r_\varepsilon/2))>\vartheta-2\varepsilon\, . 
\end{equation}
Hence by Bishop--Gromov monotonicity we deduce that, for $i$ large enough,
\[
\frac{\haus^N(B_r(p_i))}{v(N,K/(N-1),r)}>\vartheta-2\varepsilon, \qquad \text{for all $0<r\leq r_\varepsilon/2$}\, .
\]

In particular, for every $R>0$, we have that, since $\lambda_i\to+\infty$, and since $v(N,K/(N-1),r)/(\omega_Nr^N)\to 1$ as $r\to 0$, for $i$ large enough it holds
\begin{equation}\label{eq:carg2}
\frac{\haus^N_{\lambda_i\dist}(B_R^{\lambda_i\dist}(p_i))}{\omega_NR^N}=\frac{\haus^N(B_{R/\lambda_i}(p_i))}{\omega_N(R/\lambda_i)^N}>\vartheta-3\varepsilon.
\end{equation}
Hence, from \eqref{eq:carg2} and volume convergence, we get that for every $\varepsilon>0$ and every $R>0$ it holds
\[
\frac{\haus^N(B_R(x'))}{\omega_NR^N}>\vartheta-3\varepsilon\, .
\]
Taking $\varepsilon\to 0$ and $R\to +\infty$ we get the first part of the sought claim in item (2).
\medskip

For the second part it suffices to notice that $I_{C_\infty}(v)=N(\omega_N\vartheta)^{1/N}v^{\frac{N-1}{N}}$, as a consequence of \autoref{cor:IsoperimetricOnCones}, and then $I_{C_\infty}(v) \le N(\omega_N(\mathrm{AVR}(X',\dist',\haus^N)))^{1/N}v^{\frac{N-1}{N}} \le I_{X'}(v)$ by the first part of the item and the isoperimetric inequality.
\end{proof}

Below we determine the asymptotic isoperimetric behaviour for small volumes on $\RCD(K,N)$ spaces $(X,\dist,\haus^N)$ with volumes of unit balls uniformly bounded from below. Namely, we give the proof of \autoref{cor:AsymptoticProfileZero}.

\begin{proof}[Proof of \autoref{cor:AsymptoticProfileZero}]
Let us first prove that 
\begin{equation}\label{eqn:MinimumReached}
\vartheta_{\infty,\mathrm{min}}:=\inf\{\vartheta[Y,\dist_Y,\haus^N,y]:\text{$y\in Y$, $Y$ is $X$ or a pmGH limit at infinity of $X$}\},
\end{equation}
is realized. For simplicity let us call $\vartheta:=\vartheta_{\infty,\mathrm{min}}$.\\ 
It suffices to take a minimizing sequence $(Y_i,\dist_{i},\haus^N,y_i)$ such that $\vartheta[Y_i,\dist_i,\haus^N,y_i]\to \vartheta$. Since $\inf_{i\in\mathbb N}\haus^N(B_1(y_i))\ge v_0$, by volume convergence \cite[Theorem 1.2, Theorem 1.3]{DePhilippisGigli18} we can extract a subsequence such that $(Y_i,\dist_i,\haus^N,y_i)\to (Y_\infty,\dist_\infty,\haus^N,y_\infty)$ as $i\to +\infty$. A simple diagonal argument tells that $Y_\infty$ is either isometric to $X$ or to a pmGH limit at infinity of it. Moreover, from \cite[Lemma 2.2]{DePhilippisGigli18}, we get that 
\[
\vartheta[Y_\infty,\dist_\infty,\haus^N,y_\infty]\leq \liminf_{i\to +\infty}\vartheta[Y_i,\dist_i,\haus^N,y_i]=\vartheta,
\]
and thus $\vartheta[Y_\infty,\dist_\infty,\haus^N,y_\infty]=\vartheta$ showing that the infimum is attained and then it must be strictly positive.
\medskip

In order to verify \eqref{eqn:EQUA} let us first prove that 
\begin{equation}\label{eqn:Claim1}
\vartheta_{\infty,\mathrm{min}}\leq\liminf_{r\to 0}\inf_{x\in X}\frac{\haus^N(B_r(x))}{v(N,K/(N-1),r)}\, .
\end{equation}
Let us consider a decreasing sequence $r_i\downarrow 0$, such that
\[
\liminf_{r\to 0}\inf_{x\in X}\frac{\haus^N(B_r(x))}{v(N,K/(N-1),r)}=\lim_{i\to +\infty} \inf_{x\in X}\frac{\haus^N(B_{r_i}(x))}{v(N,K/(N-1),r_i)}\, .
\]
For every $i\in\mathbb N$ let us choose $x_i\in X$ such that 
\[
\frac{\haus^N(B_{r_i}(x_i))}{v(N,K/(N-1),r_i)}\leq \inf_{x\in X}\frac{\haus^N(B_{r_i}(x))}{v(N,K/(N-1),r_i)}+i^{-1}\, .
\]
Up to subsequences, $(X,\dist,\haus^N,x_i)\to (X',\dist',\haus^N,x')$ in the pmGH sense. We have, by the very definition of $\vartheta_{\infty,\mathrm{min}}$, that $\vartheta[X',\dist',\haus^N,x']\geq \vartheta_{\infty,\mathrm{min}}$.\\
Let us fix $\varepsilon>0$. Hence there exists $i_0$ sufficiently large such that for every $i\geq i_0$, 
\[
\frac{\haus^N(B_{r_{i}}(x'))}{v(N,K/(N-1),r_{i})}\geq \vartheta[X',\dist',\haus^N,x']-\varepsilon\, .
\]
By volume convergence and Bishop--Gromov comparison, for every $R<r_{i_0}$, we have that, for $i$ sufficiently large, the following holds 
\[
\frac{\haus^N(B_{r_i}(x_i))}{v(N,K/(N-1),r_i)}\geq \frac{\haus^N(B_{R}(x_i))}{v(N,K/(N-1),R)}\geq \vartheta[X',\dist',\haus^N,x']-2\varepsilon\, .
\]
Hence, for i sufficiently large, we obtain 
\[
\inf_{x\in X}\frac{\haus^N(B_{r_i}(x))}{v(N,K/(N-1),r_i)}+i^{-1}\geq \vartheta[X',\dist',\haus^N,x']-2\varepsilon
\]
and by taking $i\to +\infty$ and $\varepsilon\to 0$ in the previous, we finally obtain \eqref{eqn:Claim1}.\\
Now notice that
\begin{equation}\label{eqn:Claim2}
\vartheta_{\infty,\mathrm{min}}\geq\limsup_{r\to 0}\inf_{x\in X}\frac{\haus^N(B_r(x))}{v(N,K/(N-1),r)}\, .
\end{equation}
is just a consequence of Bishop--Gromov monotonicity, jointly with the volume convergence when $\vartheta_{\infty,\mathrm{min}}$ is reached at infinity. Joining together \eqref{eqn:Claim1} and \eqref{eqn:Claim2} we proved \eqref{eqn:EQUA}.

\medskip

Let us prove \eqref{eqn:AsymptoticIsoperimetricProfileAtZero}. We first check that \begin{equation}\label{eqn:ClaimTo}
    \liminf_{v\to 0}\frac{I(v)}{v^{\frac{N-1}{N}}}\geq N(\omega_N\vartheta)^{1/N}\, .
\end{equation}

Let us take an arbitrary sequence $V_i\to 0$ as $i\to +\infty$. We apply \autoref{thm:MassDecompositionINTRO}, together with \autoref{lem:IsoperimetricAtFiniteOrInfinite}. Up to subsequences, for every $i\in\mathbb N$, there exists $(X_i,\dist_i,\haus^N)$, which is either isometric to $(X,\dist,\haus^N)$ or to a pmGH limit at infinity of it, such that there exists an isoperimetric set $E_i\subset X_i$ with  $\haus^N(E_i)=V_i$, and $I(V_i)=I_{X_i}(V_i)=\Per (E_i)$. Let us choose arbitrary points $q_i\in E_i$. The pointed metric measure spaces $(X_i,\dist_i,\haus^N,q_i)$ converge, up to subsequences, to $(X_\infty,\dist_\infty,\haus^N,q_\infty)$ thanks to the uniform lower bound on the volume of balls and to volume convergence. By a simple diagonal process, we infer that $X_\infty$ is either isometric to $X$ or to a pmGH limit at infinity of it. 

Let us rename, for simplicity, $(X'_i,\dist_i',\haus^N_{\dist_i'},q_i'):=(X_i,V_i^{-1/N}\dist_i,V_i^{-1}\haus^N,q_i)$. Up to subsequences, arguing as above, we have
\[
(X'_i,\dist_i',\haus^N_{\dist_i'},q_i')\to (X',\dist_{X'},\haus^N,x')\, ,
\]
as $i\to +\infty$. From \autoref{lem:InequalitiesIsopProfile} we get that 
\begin{equation}\label{eqn:EstimateUseful}
\mathrm{AVR}(X',\dist_{X'},\haus^N)\geq\vartheta[X_\infty,\dist_\infty,\haus^N,q_\infty]\, .
\end{equation}

Moreover, $\haus^N_{\dist_i'}(E_i)=1$ and, by \autoref{prop:NewRegularityIsoperimetricSets}, $\diam_{\dist_i'}E_i\leq k$, where $k$ only depends on $N,K,v_0$. Hence we can apply \cite[Theorem 3.3]{AmbrosioBrueSemola19} to get that, up to subsequences, $E_i\subset (X'_i,\dist_i',\haus^N_{\dist_i'},q_i')$ converge in $L^1$-strong to a finite perimeter set $E_\infty\subset X'$ with $\haus^N(E_\infty)=1$. Moreover, from \cite[Proposition 3.6]{AmbrosioBrueSemola19} we get that 
\begin{equation}\label{eqn:EvvaiEvvai}
\begin{split}
\Per(E_\infty)&\leq \liminf_{i\to +\infty}\Per(E_i)=\liminf_{i\to +\infty}V_i^{-\frac{N-1}{N}}\Per(E_i)\\
&= \liminf_{i\to +\infty}V_i^{-\frac{N-1}{N}}I(V_i)\, .
\end{split}
\end{equation}
From the sharp isoperimetric inequality \eqref{eqn:SharpIsopintro} applied on $(X',\dist_{X'},\haus^N)$ and \eqref{eqn:EstimateUseful}, we have 
\begin{equation}\label{eqn:EvvaiEvvai2}
\begin{split}
    \Per(E_\infty)&\geq N(\omega_N\mathrm{AVR}(X',\dist_{X'},\haus^N))^{1/N}\haus^N(E_\infty)^{\frac{N-1}{N}}\geq N(\omega_N\vartheta[X_\infty,\dist_\infty,\haus^N,q_\infty])^{1/N}\\
    &\geq N(\omega_n\vartheta)^{1/N}\, .
\end{split}
\end{equation}
Joining together \eqref{eqn:EvvaiEvvai} and \eqref{eqn:EvvaiEvvai2}, we conclude that
\begin{equation}\label{eqn:DAUNIRE}
    \liminf_{i\to +\infty}V_i^{-\frac{N-1}{N}}I(V_i)\geq \Per(E_\infty)
    \geq N(\omega_N\vartheta[X_\infty,\dist_\infty,\haus^N,q_\infty])^{1/N}\geq  N(\omega_n\vartheta)^{1/N},
\end{equation}
and thus we obtain
\eqref{eqn:ClaimTo}.

Let us now prove that 
\begin{equation}\label{eqn:ClaimTo2}
    \limsup_{v\to 0}\frac{I(v)}{v^{\frac{N-1}{N}}}\leq N(\omega_N\vartheta)^{1/N}\, .
\end{equation}
Let us take $(Y,\dist_Y,\haus^N,y)$, with $Y$ isometric to $X$ or to some pmGH limit at infinity of it, such that $\vartheta[Y,\dist_Y,\haus^N,y]=\vartheta$.
Hence, by \cite[Proposition 2.18]{AntonelliNardulliPozzetta} we have that $I(v)\leq I_Y(v)$ for every $v\geq 0$. Moreover, from the first item of \autoref{lem:InequalitiesIsopProfile}, we get that 
\[
\lim_{v\to 0}\frac{I_Y(v)}{v^{(N-1)/N}}\leq N(\omega_N\vartheta)^{1/N}.
\]
Putting the last two inequalities together, we get 
\eqref{eqn:ClaimTo2}. Then \eqref{eqn:AsymptoticIsoperimetricProfileAtZero} follows taking also \eqref{eqn:ClaimTo} into account.
\medskip

The proof of the second item readily comes from the proof of the first item above. Indeed, from \eqref{eqn:DAUNIRE} and \eqref{eqn:ClaimTo2} we conclude that all the inequalities in \eqref{eqn:DAUNIRE} are equalities and thus $\vartheta[X_\infty,\dist_\infty,\haus^N,q_\infty]=\vartheta$. Hence $q_\infty$ is a point at which the minimum in \eqref{eqn:MinimumReached} is reached. Moreover, since $\diam_{\dist_i}E_i\to 0$, due to the fact that $\haus^N(E_i)\to 0$ and (2) in \autoref{prop:NewRegularityIsoperimetricSets}, we get that the sets $E_i$ converge to $q_\infty$ in Hausdorff distance in a realization.

\medskip

For the proof of the third item we exploit the proof of the first item above. Indeed, from \eqref{eqn:DAUNIRE} and \eqref{eqn:ClaimTo2} we conclude that all the inequalities in \eqref{eqn:EvvaiEvvai} and \eqref{eqn:EvvaiEvvai2} are equalities. Hence we have that 
\[
\Per(E_\infty)=N(\omega_N\mathrm{AVR}(X',\dist_{X'},\haus^N))^{1/N}\haus^N(E_\infty)^{\frac{N-1}{N}}\, ,
\]
thus $E_\infty$ saturates the sharp isoperimetric inequality \autoref{thm:IsoperimetricSharpRigidintro} on $X'$. Hence, by the rigidity part of \autoref{thm:IsoperimetricSharpRigidintro}, $X'$ is a Euclidean metric measure cone over an $\RCD(N-2,N-1)$ space, its opening is equal to $\mathrm{AVR}(X',\dist_{X'},\haus^N)=\vartheta[X_\infty,\dist_\infty,\haus^N,q_\infty]=\vartheta$ and $E_{\infty}$ is a ball centred at one of the tips of $X'$. Finally, applying \autoref{thm:ConvergenceBarriersStability}, we deduce that the convergence of $\overline{E}_i \subset X_i'$ the the balls $E_\infty$ holds in Hausdorff sense.
\end{proof}

\begin{remark}
If $(X,\dist,\haus^N)$ is a compact $\RCD(K,N)$ metric measure space, then the statement of \autoref{cor:AsymptoticProfileZero} simplifies, since it is not necessary to consider pointed limits at infinity and $\vartheta_{\infty, \mathrm{min}}$ is the minimal density at a point $p\in X$, which is attained by lower semicontinuity of the density and compactness. Moreover, isoperimetric regions of small volume converge in the Hausdorff sense and up to subsequences to a point where the minimal density is realized.
\end{remark}

\begin{remark}\label{rm:exmin1}
In \eqref{eqn:AsymptoticIsoperimetricProfileAtZero}, it may occur that $\vartheta_{\infty,\min}<1$ also on smooth Riemannian manifolds with Ricci curvature bounded below and volume of unit balls bounded below. For example, let $M=\R \times [0,+\infty) \times \mathbb{S}^1$ be endowed with a Riemannian metric $g$ of the form $g= {\rm d}t^2 + {\rm d}r^2 + \sigma(t,r)^2 {\rm d}\theta^2$, where $t \in \R,$ $r \in [0,+\infty)$, $\theta \in \mathbb{S}^1$ and $(\mathbb S^1,{\rm d} \theta^2)$ is the circle of radius $1$. For any $t \in \R$, we can arrange
\[
\sigma(t,r) = \begin{cases}
r & \text{if } r \in [0,R_{1,t}] ,\\
h(t)+\frac{r}{2} & \text{if }  r \in [R_{2,t},+\infty) ,
\end{cases}
\]
where $0<R_{1,t}<R_{2,t} \to 0$ and $0<h(t)\to 0$ as $|t|\to+\infty$, and $r\mapsto \sigma(t,r)$ is concave for any $t$. Then $(M,g)$ is smooth and without boundary. Moreover, a direct computation of the Ricci curvature on the frame $\partial_t, \partial_r, \partial_\theta / \sigma$ yields
\[
\begin{aligned}[t]
    &\Ric(\partial_t, \partial_t) = - \frac{\partial_t^2 \sigma}{\sigma},\\
    %\qquad
    &\Ric(\partial_r, \partial_r) = - \frac{\partial_r^2 \sigma}{\sigma}, 
\end{aligned}
\qquad
\begin{aligned}[t]
&   \Ric\left(\frac{\partial_\theta}{\sigma}, \frac{\partial_\theta}{\sigma} \right) = - \frac{\partial_t^2 \sigma}{\sigma} - \frac{\partial_r^2 \sigma}{\sigma},\\
    %\qquad
    &\Ric(\partial_t, \partial_r) = - \frac{\partial_t\partial_r\sigma}{\sigma},
\end{aligned}
\]
and $\Ric(\partial_t,  \partial_\theta / \sigma) = \Ric(\partial_r,  \partial_\theta / \sigma)=0$. By concavity with respect to $r$, we ensure that $ - \tfrac{\partial_r^2 \sigma}{\sigma}\ge 0$, and by taking $t\mapsto R_{1,t},R_{2,t}$ varying sufficiently slowly, we can set $\Ric \ge K$ on $M$ for some $K\le0$.

On the other hand, $(M,g)$ looks like a smoothing of the cone $C\eqdef (M, {\rm d}t^2 + {\rm d}r^2 + (r/2)^2 {\rm d}\theta)$ where the smoothing effect worsens as $|t|\to+\infty$. Indeed, the pmGH limit of $(M,g)$ along a sequence of the form $(t_i,0,q)$ for $t_i\to+\infty$ is the cone $C$. Hence $\vartheta_{\infty,\min}$ in this case is achieved by the density at any point $(t,0,q) \in C$, and it is strictly less than $1$.
\end{remark}

\begin{remark}
The first two items of \autoref{cor:AsymptoticProfileZero} generalize \cite[Theorem 6.9]{LeonardiRitore} to the case of $\RCD(K,N)$ spaces with uniform lower bounds on the volume of unit balls. Notice that the generalization is strict, since the results \cite{LeonardiRitore} are stated for convex bodies in $\mathbb R^N$ with a uniform bound below on the volume of unit balls. %In particular, our result is new for smooth, non compact manifolds with Ricci curvature bounded from below and for non compact Alexandrov spaces with sectional curvature bounded from belw.

Moreover, the third item of \autoref{cor:AsymptoticProfileZero} partially generalizes
\cite[Corollary 6.15]{LeonardiRitore} to the case of $\RCD(K,N)$ spaces with uniform lower bounds on the volume of unit balls. Again the generalization is non-trivial, since \cite[Corollary 6.15]{LeonardiRitore} holds for convex bodies in $\mathbb R^N$ with a uniform bound below on the volume of unit balls. %In particular, as above, the result is new for smooth, non compact manifolds with Ricci curvature bounded from below and for Alexandrov spaces with sectional curvature bounded from below.

For the sake of comparison, let us point out that \cite[Corollary 6.15]{LeonardiRitore} is more precise than item (3) of \autoref{cor:AsymptoticProfileZero}: it states that the rescaled sets converge in the Hausdorff sense precisely in the tangent cone at a point where the minimal opening is reached. Instead, in \autoref{cor:AsymptoticProfileZero} we do not prove that the convergence of the rescaled sets holds in a tangent cone at a point where the minimal opening is reached, and the validity of such a statement in the present setting goes beyond the scope of this note and is left to future investigation.
\end{remark}

We consider some consequences of the asymptotic of the isoperimetric profile for small volumes (\autoref{cor:AsymptoticProfileZero}).
\medskip

% \begin{remark}[Improving \autoref{cor:EstimateDerivativeProfileAndSubadditivity}]\label{rem:ImproveDerivativeProfileSubadditivity}
% Let $(X,\dist,\haus^N)$ be an $\RCD(K,N)$ space with isoperimetric profile function $I$. Let us assume $\inf_{x\in X}\haus^N(B_1(x))\geq v_0>0$. By \autoref{cor:AsymptoticProfileZero}, the limit $\vartheta$ in \autoref{cor:EstimateDerivativeProfileAndSubadditivity} is now known to be equal to $N(\omega_N \vartheta_{\infty,\min})^{\frac1N}$. Hence
% \begin{equation}
%     \lim_{v\to0} \frac{I'_+(v)}{v^{-\frac1N}} = (N-1) (\omega_N \vartheta_{\infty,\min})^{\frac1N}\, .
% \end{equation}
% \end{remark}

\medskip

Let us introduce the normalized isoperimetric profile $\bar{I}:[0,1]\to [0,\infty]$ of a metric measure space $(X,\dist,\meas)$ with finite measure by 
\begin{equation}
\bar{I}(v):=I(\meas(X)v)\, ,
\end{equation}
where $I:[0,\meas(X)]\to[0,\infty]$ is the isoperimetric profile of $(X,\dist,\meas)$.

\begin{proposition}\label{prop:unifrapp}
Let $(X_n,\dist_n,\haus^N)$ be $\RCD(K,N)$ spaces with diameters uniformly bounded from above by $0<D<\infty$, assume that $\inf_{n\in\mathbb N}\haus^N(X_n)>0$, and assume that they converge to $(X,\dist,\haus^N)$ in the Gromov--Hausdorff sense. Then 
\begin{equation}\label{eq:unifquot}
    \lim_{n\to\infty}\sup_{v\in [0,1]}\left |\frac{\bar{I}_n(v)}{\bar{I}(v)}-1\right|=0\, ,
\end{equation}
if and only if 
\begin{equation}\label{eq:convmindens}
\lim_{n\to\infty}\min_{x\in X_n}\vartheta[X_n,\dist_n,\haus^N,x]=\min_{x\in X}\vartheta[X,\dist,\haus^N,x]\, .
\end{equation}
\end{proposition}

\begin{proof}
The pointwise convergence of the normalized isoperimetric profiles under these assumptions is well known and it follows from the convergence and stability of functions of bounded variation and sets of finite perimeter, see \cite{AmbrosioHonda17}.
\medskip

The pointwise convergence can be strengthened to locally uniform convergence on compact subsets $[\alpha,\beta]\Subset [0,1]$ thanks to the uniform Lipschitz property of the isoperimetric profiles, see \autoref{cor:FinePropertiesProfile}.
Then the uniform convergence to $1$ of the ratios $\bar{I}_n/\bar{I}$ on compact subsets of $(0,1)$ follows since $\bar{I}$ is locally uniformly bounded away from $0$ on $(0,1)$ under our assumptions.
\medskip

The implication from \eqref{eq:unifquot} to \eqref{eq:convmindens} follows from the explicit asymptotic behaviour of the isoperimetric profile for small volumes in terms of the minimal density \autoref{cor:AsymptoticProfileZero}.

The implication from \eqref{eq:convmindens} to \eqref{eq:unifquot} follows from item (2) of \autoref{cor:FinePropertiesProfile}, \autoref{cor:AsymptoticProfileZero} and \cite[Lemma B.3.4]{Bayle03}.
\end{proof}

\begin{remark}
We point out that the inequality $\geq$ holds true unconditionally in \eqref{eq:convmindens}, by lower semicontinuity of the density (see \cite[Lemma 2.2]{DePhilippisGigli18}).
% \begin{equation}
% \liminf_{n\to\infty}\min_{x\in X_n}\vartheta[X_n,\dist_n,\haus^N,x]\ge \min_{x\in X}\vartheta[X,\dist,\haus^N,x]\, .
% \end{equation}
\end{remark}

\begin{remark}
Let us point out that \autoref{prop:unifrapp} gives a fairly complete answer to the questions raised at the end of \cite[Remark 4.3.4]{Bayle03}. Indeed it completely characterizes the uniform convergence to $1$ of the quotients of the isoperimetric profiles $I_n/I$ in terms of the convergence of the minimal densities. In particular, the uniform convergence of the quotients holds whenever $(X_n,\dist_n,\haus^N)$ are smooth Riemannian manifolds and $(X,\dist,\haus^N)$ is a smooth Riemannian manifold, since under these assumptions
\begin{equation}
\vartheta[X_n,\dist_n,\haus^N,x_n]=\min_{x\in X}\vartheta[X,\dist,\haus^N,x]\,, \quad\text{for any $x_n\in X_n$ and any $x\in X$}\, .
\end{equation}
However, even under the assumption that $(X_n,\dist_n,\haus^N)$ are smooth Riemannian manifolds, it is not necessary that the limit is a smooth Riemannian manifold for \eqref{eq:unifquot} to hold. Indeed, there are elementary examples of $\RCD(K,N)$ spaces $(X,\dist,\haus^N)$ with empty singular set that are not smooth Riemannian manifolds.
\end{remark}

\subsection{Isoperimetric almost regularity theorems}

It is a classical result, pointed out for instance in \cite{Ledouxoptimal}, that a complete Riemannian manifold $(M^n,g)$ with non negative Ricci curvature such that the sharp Euclidean isoperimetric inequality holds is isometric to $\setR^n$. The result was strengthened in \cite{Xiaalmost}, where it is proved that if an almost Euclidean isoperimetric inequality holds, then $(M^n,g)$ is diffeomorphic to $\setR^n$. Both proofs are based on the observation that the isoperimetric inequality controls from below the volume of balls. In particular, a Euclidean isoperimetric inequality forces a Euclidean behaviour of the volume of balls. Therefore rigidity in the Bishop--Gromov inequality holds and the manifold is isometric to $\setR^n$. Analogously, the almost Euclidean isoperimetric inequality forces almost Euclidean lower bounds on the volume of balls, and the statement in \cite{Xiaalmost} follows from Cheeger--Colding's Reifenberg's theorem \cite{ChCo1}.

Conversely, it is known after \cite{CavallettiMondinoalmost} that almost Euclidean lower volume bounds imply almost Euclidean isoperimetric inequalities on smaller balls see also the recent \cite[Theorem 3.9]{NobiliViolo}.

As a consequence of our main results, an almost Euclidean isoperimetric inequality for a given volume is sufficiently strong to guarantee almost-regularity. Moreover, an almost Euclidean lower bound on the volume of balls forces an almost Euclidean isoperimetric inequality for small volumes. On the one hand this statement is stronger than the almost Euclidean isoperimetric inequality inside small balls from \cite{CavallettiMondinoalmost}, since small diameter implies small volume by the Bishop--Gromov inequality. On the other hand, the assumptions in \cite{CavallettiMondinoalmost} are local and more general (no assumptions on infinitesimal Hilbertianity nor on non collapsing), while we need global conditions in our argument.

\begin{theorem}\label{thm:isoperimetricepsregularity}
Let $(X,\dist,\haus^N)$ be an $\RCD(K,N)$ metric measure space. Then there exists a constant $C(K,N)>0$ such that for any $\eps>0$ there exist $\delta(\eps,K,N)>0$ and $v(\eps,K,N)>0$ for which the following holds: if there exists $0<v<v(\eps,K,N)$ such that 
\begin{equation}\label{eq:almostiso}
\frac{I(v)}{v^{\frac{N-1}{N}}}\ge N\omega_N^{\frac{1}{N}}-\delta\, ,
\end{equation}
then for any $0<r<C(K,N)v^{\frac{1}{N}}$ and any $x\in X$ it holds 
    \begin{equation}\label{eq:isoalmvol}
        \frac{\haus^N(B_r(x))}{\omega_Nr^N}\ge 1-\eps\, .
    \end{equation}
Conversely, for any $\eps>0$ there exist constants $r(\eps,K,N)>0$, $v(\eps,r,K,N)$ and $\delta(\eps,K,N)>0$
for which the following holds: if $(X,\dist,\haus^N)$ is an $\RCD(K,N)$ metric measure space and there exists $0<r<r(\eps,K,N)$ such that 
\begin{equation}\label{eq:almeuclvol}
     \frac{\haus^N(B_r(x))}{\omega_Nr^N}\ge 1-\delta\, , \quad\text{for any $x\in X$}\, ,
\end{equation}
then 
\begin{equation}
\frac{I(v)}{v^{\frac{N-1}{N}}}\ge N\omega_N^{\frac{1}{N}}-\eps\, ,\quad \text{for any $0<v<v(\eps,r,K,N)$}\, .
\end{equation}
    
\end{theorem}

\begin{proof}
We prove the first part of the statement under the assumption that $K=0$, in which case it is possible to choose 
\begin{equation}
v(\eps,0,N)=+\infty\, , \quad C(0,N)=\frac{1}{2}\omega_{N}^{\frac{1}{N}}\,.
\end{equation}
The case $K<0$ can be handled with minor modifications with respect to the argument that we are going to present, as in the case of previous arguments in the note.
\medskip

If $K=0$ and \eqref{eq:almostiso} holds for some $v>0$, then by \eqref{eq:monotonicityformula}
\begin{equation}\label{eq:almeuclall}
 \frac{I(\bar{v})}{\bar{v}^{\frac{N-1}{N}}}\ge N\omega_N^{\frac{1}{N}}-\delta\, ,  \quad\text{for any $0<\bar{v}<v$}\, .
\end{equation}
We wish to apply the almost Euclidean isoperimetric inequality \eqref{eq:almeuclall} to balls with small radii, to estimate their volume via the coarea formula.\\
Let us consider any $x\in X$ and any $0<s<C(0,N)v^{\frac{1}{N}}$. Notice that 
\begin{equation}
    \haus^N(B_s(x))\le v\, ,
\end{equation}
by the Bishop--Gromov inequality and the very definition of $C(0,N)$. Hence, $B_s(x)$ is a competitor for the isoperimetric inequality in the range where \eqref{eq:almeuclall} holds. Therefore it holds
\begin{equation}\label{eq:ale}
\Per(B_s(x))\ge \left( N\omega_N^{\frac{1}{N}}-\delta\right)\left(\haus^N(B_s(x))\right)^{\frac{N-1}{N}}\, ,\quad\text{for any $0<s<C(0,N)v^{\frac{1}{N}}$} \, .
\end{equation}
Let us set $f(s):=\Per(B_s(x))$. The coarea formula implies that
\begin{equation}
\haus^N(B_r(x))=\int_0^rf(s)\di s\, , \quad\text{for any $r>0$}\, .
\end{equation}
Therefore, \eqref{eq:ale} can be turned into
\begin{equation}
f(r)\ge \left( N\omega_N^{\frac{1}{N}}-\delta\right)\left(\int_0^rf(s)\di s\right)^{\frac{N-1}{N}}\, ,\quad\text{for any $r>0$}\, .
\end{equation}
By integrating the integral inequality above we easily infer that \eqref{eq:isoalmvol} holds whenever $x$ is a regular point. The statement follows since the function $x\mapsto \haus^N(B_r(x))$ is continuous, by Bishop--Gromov inequality, and the regular set is dense.
\medskip

Let us prove the converse implication, focusing again on the case $K=0$. The general case can be handled with minor modifications.\\ 
If $K=0$ then we can set $r(\eps,K,N)=+\infty$. Notice that, by volume convergence, the lower volume bound in \eqref{eq:almeuclvol} holds also for any pointed limit at infinity $(Y_{\infty},\dist_{\infty},\haus^N,y_{\infty})$ of $(X,\dist,\haus^N)$. In particular, by Bishop--Gromov volume montonicity, all the densities at any point in any pointed limit at infinity of $(X,\dist,\haus^N)$ are bounded from below by $1-\delta$. The almost Euclidean isoperimetric inequality for small volumes $0<v<v((X,\dist,\haus^N))$ follows from \eqref{eqn:AsymptoticIsoperimetricProfileAtZero}.\\
Let us show that actually the conclusion holds for $0<v<v(\eps,r,0,N)$, with $v(\eps,r,0,N)$ independent of the metric measure space $(X,\dist,\haus^N)$ verifying the assumptions of the statement.\\
If this is not the case, then we can find a sequence of $\RCD(0,N)$ metric measure spaces $(X_n,\dist_n,\haus^N)$ such that 
\begin{equation}\label{eq:unilow}
    \frac{\haus^N(B^{X_n}_r(x_n))}{\omega_Nr^N}\ge 1-\frac{1}{n}\, ,
\end{equation}
for any $x_n\in X_n$ for any $n\in\setN$ and Borel sets $E_n\subset X_n$ such that 
\begin{equation}
\haus^N(E_n)\downarrow 0\, ,\quad\Per(E_n)\le \left(N\omega_N^{\frac{1}{N}}-\eps\right)\left(\haus^N(E_n)\right)^{\frac{N-1}{N}}\, .
\end{equation}
By volume convergence, any pointed limit at infinity $(Y,\dist_Y,\haus^N,y)$ of any of the metric measure spaces $(X_n,\dist_n,\haus^N)$ verifies \eqref{eq:unilow}. Therefore, up to possibly substituting $(X_n,\dist_n,\haus^N)$ with a pointed limit at infinity, we can assume thanks to \autoref{lem:IsoperimetricAtFiniteOrInfinite} that $E_n\subset X_n$ is an isoperimetric set.\\
Thanks to (2) in \autoref{prop:NewRegularityIsoperimetricSets}, $\diam E_n\le C\haus^N(E_n)^{\frac{1}{N}}$, for a uniform constant $C(N)$, for any sufficiently large $n\in\setN$. After a point picking and scaling argument, we obtain a sequence of metric measure spaces $(Y_n,\dist_n,\haus^N,y_n)$ converging in the pmGH topology to $\setR^N$ with canonical metric measure structure and sets $F_n\subset B_2(y_n)$ with measures uniformly bounded and uniformly bounded away from zero such that
\begin{equation}
\Per(F_n)\le \left(N\omega_N^{\frac{1}{N}}-\eps\right)\left(\haus^N(F_n)\right)^{\frac{N-1}{N}}\, .
\end{equation}
Up to the extraction of a subsequence, the sets $F_n$ converge in $L^1$ strong to $F\subset B_2(0^N)\subset \setR^N$. By lower semicontinuity of the perimeter under $L^1$ strong convergence, this contradicts the Euclidean isoperimetric inequality.
\end{proof}

\begin{remark}\label{rm:Reif}
Under the same assumptions of \autoref{thm:isoperimetricepsregularity} above, if $\eta>0$ and \eqref{eq:isoalmvol} holds for some $\eps<\eps(\eta,K,N)$ and some $r<r(\eta,K,N)$ and $x\in X$, then
\begin{equation}
\dist_{GH}\left(B_s(x),B_s(0^N)\right)<\eta s\, ,\quad\text{for any $0<s<r/2$}\, .
\end{equation}
This is a consequence of the classical almost volume rigidity theorem \cite{Anderson90,Colding97,ChCo1,DePhilippisGigli18}. 

From this observation and Reifenberg's theorem \cite{ChCo1,MondinoKapovitch} it follows that there exists $\eps=\eps(N)>0$
such that if $(X,\dist,\haus^N)$ is an $\RCD(K,N)$ metric measure space and 
\begin{equation}
    \lim_{v\to 0}\frac{I(v)}{v^{\frac{N-1}{N}}}\ge N\omega_N^{\frac{1}{N}}-\eps(N)\, ,
\end{equation}
then $(X,\dist,\haus^N)$ and all its pointed limits at infinity are homeomorphic to smooth Riemannian manifolds. Furthermore, if 
\begin{equation}
    \lim_{v\to 0}\frac{I(v)}{v^{\frac{N-1}{N}}}= N\omega_N^{\frac{1}{N}}\,,
\end{equation}
then $(X,\dist,\haus^N)$ and all its pointed limits at infinity have empty singular set.
\end{remark}

\begin{remark}
The very same argument presented for the proof of the second implication in \autoref{thm:isoperimetricepsregularity}, together with the explicit asymptotics of the isoperimetric profile for small volumes \eqref{eqn:AsymptoticIsoperimetricProfileAtZero}, shows that the following holds: if $(X,\dist,\haus^N)$ is an $\RCD(K,N)$ metric measure space such that 
\begin{equation}
\frac{\haus^N(B_r(x))}{v_{K,N}(r)}\ge \alpha\, ,\quad\text{for any $x\in X$}\, ,
\end{equation}
for some $1\ge \alpha>0$ and some $r>0$, then for any $\eps>0$ there exists $v_{\eps}:=v_\eps(\eps,K,N,\alpha)>0$ such that an \emph{almost conical} isoperimetric inequality 
\begin{equation}
I(v)\ge \left(N\omega_{N}^{\frac{1}{N}}\alpha-\eps\right)v^{\frac{N-1}{N}}
\end{equation}
holds for any $0<v<v_{\eps}$.
\end{remark}

We wish to specialize the isoperimetric almost regularity theorem to the case of non collapsing manifolds with two-sided Ricci curvature bounds and Einstein manifolds. 

Let us recall that the regular set of a noncollapsed limit of Riemannian manifolds with uniformly bounded Ricci curvature is an open set, isometric to a $C^{1,\alpha}$-Riemannian manifold for any $0<\alpha<1$, see \cite{Anderson90,ChCo1}.

\begin{definition}
Let $(X,\dist)$ be a noncollapsed limit of smooth $n$-dimensional Riemannian manifolds with Ricci curvature uniformly bounded from below. Given any $x\in X$ we define the harmonic radius $r_h(x)$ so that $r_h(x)=0$ if there is no neighbourhood of $x$ where $(X,\dist)$ is isometric to a Riemannian manifold $(M,g)$. Otherwise we define $r_h(x)$ to be the largest $r>0$ such that there exists a mapping $\Phi:B_r(0^n)\subset \mathbb R^n\to X$ such that
\begin{itemize}
    \item $\Phi(0)=x$ and $\Phi$ is a diffeomorphism with its image;
    \item $\Delta_gx^\ell=0$, where $x^\ell$ are the coordinate functions and $\Delta_g$ is the Laplace--Beltrami operator;
    \item if $g_{ij}=\Phi^*g$ is the pullback metric, then 
    \begin{equation}
        ||g_{ij}-\delta_{ij}||_{C^0(B_r(0^n))}+r||\partial_kg_{ij}||_{C^0(B_r(0^n))}<10^{-3}\, .
    \end{equation}
\end{itemize}
\end{definition}

In the case of noncollapsed limit of Einstein manifolds with uniformly bounded Einstein constants, the regular set is isometric to a $C^{\infty}$-Riemannian manifold, see \cite{Anderson90,ChCo1}.

\begin{definition}
Let $(X,\dist)$ be a noncollapsed limit of smooth $n$-dimensional Einstein manifolds with uniformly bounded Einstein constants. Given any $x\in X$ we define the regularity scale $r_x$ so that $r_x=0$ if $x$ is a singular point and 
\begin{equation}
    r_x:=\max_{0<r\le 1}\left \{ \sup_{B_r(x)}|\mathrm{Riem}|\le r^{-2}\right\}\, .
\end{equation}
\end{definition}

\begin{corollary}\label{cor:RicciflatEinstein}
Let $n\ge 2$ be fixed. Then there exists $\eta=\eta(n)>0$ such that the following holds. Let $(M^n,g)$ be a smooth and complete Riemannian manifold with bounded Ricci curvature. Then the harmonic radius is uniformly bounded away from zero on $M$ if and only if 
\begin{equation}\label{eq:harmonicrad}
    \lim_{v\to 0} \frac{I(v)}{v^{\frac{n-1}{n}}}\ge n\omega_n^{\frac{1}{n}}-\eta\, .
\end{equation}
Moreover, if \eqref{eq:harmonicrad} holds, then 
\begin{equation}\label{eq:imp}
    \lim_{v\to 0} \frac{I(v)}{v^{\frac{n-1}{n}}}= n\omega_n^{\frac{1}{n}}\,  
\end{equation}
and all the pointed limits at infinity of $(M,g)$ are Riemannian manifolds with $C^{1,\alpha}\cap W^{2,q}$ Riemannian metric for any $\alpha<1$ and any $q<\infty$.

If the manifold $(M^n,g)$ is Einstein, the same conclusion holds with the regularity scale in place of the harmonic radius and $C^{1,\alpha}\cap W^{2,q}$ replaced by $C^{\infty}$.
\end{corollary}

\begin{proof}
The statement follows from \autoref{thm:isoperimetricepsregularity} thanks to the $\eps$-regularity theorems for Einstein manifolds and manifolds with bounded Ricci curvature from \cite{Anderson90,ChCo1}: there exists $\eps(n,v)>0$ such that if $(M^n,g)$ satisfies $|\mathrm{Ric}|\le \eps$ and $\mathrm{vol}(B_1(p))>v$ and
\begin{equation}
    \dist_{GH}(B_2(p),B_2(0^n))<\eps\, ,
\end{equation}
then $r_h(p)\ge 1$.
Moreover, if $(M^n,g)$ is Einstein, then $r_p\ge 1$.
\medskip

We prove the statement in the case of bounded Ricci curvature, the case of Einstein manifolds being completely analogous.\\
If \eqref{eq:harmonicrad} holds for $\eta=\eta(\eps,n)$ small enough, then by \autoref{thm:isoperimetricepsregularity} and \autoref{rm:Reif} all the tangent cones of all the pointed limits at infinity of $(M^n,g)$ have density bigger than $1-\eps$ and unit balls $\eps$-GH close to $B_1(0^n)\subset \setR^n$. By the $\eps$-regularity theorem that we recalled above, all the pointed limits at infinity have empty singular set. Then \eqref{eq:imp} follows from \autoref{cor:AsymptoticProfileZero} (1).\\
For the very same reasons, there exists $r>0$ such that 
\begin{equation}
    \dist_{GH}(B_r(p),B_r(0^n))<\eps r\, ,
\end{equation}
for any $p\in M$. The regularity of the pointed limits at infinity and the uniform lower bound on the harmonic radius follow again from the $\eps$-regularity theorem.
\end{proof}

\subsection*{Acknowledgements}
The last author is grateful to Elia Bru\`e and Andrea Mondino for several conversations and useful comments on a preliminary version of this note and to Christian Ketterer for useful comments.
The authors are grateful to Otis Chodosh, Mattia Fogagnolo and Emanuel Milman for useful feedbacks on a preliminary version of the paper {and to the reviewer for careful reading and valuable comments on a previous version of the note.}

\subsection*{Declarations}
Conflict of interest. The authors state that there is no conflict of interest.

\printbibliography[title={References}]

\typeout{get arXiv to do 4 passes: Label(s) may have changed. Rerun} %Questo comando pare che sia per far funzionare poi la compilazione su arXiv

\end{document}